\documentclass[11pt,reqno]{amsart}
\usepackage{amsmath, amssymb, graphicx, mathrsfs, hyperref}
\usepackage{amsthm}
\usepackage{palatino,bbm}
\usepackage{autonum}
\usepackage[makeroom]{cancel}
\usepackage{enumitem}  
\usepackage[usenames,dvipsnames,svgnames]{xcolor}
\usepackage[top=1.2in, bottom=1in, left=1in, right=1in]{geometry} 
\allowdisplaybreaks

\renewcommand{\Im}{\operatorname{Im}}

\newcommand{\Res}{\operatorname{Res}}
\renewcommand{\Re}{\operatorname{Re}}
\renewcommand{\Im}{\operatorname{Im}}
\newcommand{\s}{{\sigma}}

\renewcommand{\a}{\alpha}
\renewcommand{\b}{\beta}
\newcommand{\e}{\epsilon}
\renewcommand{\d}{{\delta}}
\newcommand{\g}{\gamma}
\newcommand{\G}{\Gamma}

\renewcommand{\o}{\over}

\newcommand{\bs}{\boldsymbol}

\renewcommand{\(}{\left\(}
\renewcommand{\)}{\right\)}
\renewcommand{\[}{\left\[}
\renewcommand{\]}{\right\]}
\renewcommand{\i}{\infty}

\numberwithin{equation}{section}
\theoremstyle{plain}
\newtheorem{theorem}{Theorem}[section]
\newtheorem{lemma}[theorem]{Lemma}

\newtheorem*{remark*}{Remark}

\newtheorem{corollary}[theorem]{Corollary}

\makeatletter
\def\proof{\@ifnextchar[{\@oproof}{\@nproof}}
\def\@oproof[#1][#2]{\trivlist\item[\hskip\labelsep\textit{#2 Proof of\
		#1.}~]\ignorespaces}
\def\@nproof{\trivlist\item[\hskip\labelsep\textit{Proof.}~]\ignorespaces}

\makeatother

\begin{document}
	\title[Riesz type criteria for $L$-functions in the Selberg class]{Riesz type criteria for $L$-functions in the Selberg class} 
	
	\author{Shivajee Gupta and Akshaa Vatwani}
	
	\address{Discipline of Mathematics, Indian Institute of Technology Gandhinagar, Palaj, Gandhinagar 382355, Gujarat, India} 
	\email{shivajee.o@iitgn.ac.in, akshaa.vatwani@iitgn.ac.in \newline }
	\begin{abstract}
	We formulate a generalization of Riesz-type criteria in the setting of  $L$-functions belonging to the Selberg  class.   We obtain a  criterion which is sufficient for the Grand Riemann Hypothesis (GRH) for $L$-functions satisfying axioms of the Selberg class without imposing the Ramanujan hypothesis on their coefficients. We also construct a subclass of the Selberg class and  prove a necessary criterion for GRH for $L$-functions in this subclass.  Identities of Ramanujan-Hardy-Littlewood type  are also established   in this setting, specific cases of which  yield new transformation formulas involving special values of the Meijer $G$-function of the type $G^{n \  0}_{0 \ n}$. 
	\end{abstract}

	\maketitle

\section{Introduction}
The well-known Riemann Hypothesis asserts that all non-trivial zeros of the Riemann zeta function $\zeta(s)$ lie on the line $\Re(s) =1/2$. 
In 1916, Riesz  \cite{riesz1916} showed  that a necessary and sufficient criterion for the Riemann Hypothesis is  the bound 
\begin{align} \label{riesz criteria}
\sum_{n=1}^\infty \mu(n)  \frac{x}{n^2} e^{-x/n^2} =O_\d \big(  x^{\frac{1}{4}+\d} \big),  
\end{align}
for any $\d>0$. Around  the same time, Hardy and Littlewood \cite{hl} established that the bound 
\begin{align} \label{hl criteria}
\displaystyle\sum\limits_{n=1}^{\infty}\frac{\mu(n)}{n}e^{-x/n^2} = O_{\delta}\big(x^{-\frac{1}{4}+\delta}\big)
\end{align}
for any $\d>0$ is equivalent to the Riemann Hypothesis.  Various variants of Riesz-type criteria have been given in the literature, for instance for Dirichlet $L$-functions  and in the setting of primitive Hecke forms  by Dixit, Roy and Zaharescu in  \cite{riesz} and \cite{hecke} respectively; by Dixit, Gupta and Vatwani \cite{dgv} for the Dedekind zeta function; and by Banerjee and Kumar \cite{banerjee-kumar} for $L$-functions associated to primitive Maass cusp forms over the congruence subgroup $\Gamma_0(N)$. In \cite{krr}, K\"{u}hn, Robles and Roy obtained  a generalized  Riesz-type  criterion involving functions which are reciprocal under a certain Hankel transformation. Their result holds for $L$-functions in the reduced Selberg class of degree one.   Recently, Agarwal, Maji and Garg \cite{agm} obtained a generalization of \eqref{riesz criteria} and \eqref{hl criteria} by giving criteria involving the sum 
$
\sum\limits_{n=1}^{\infty}\frac{\mu(n)}{n^k}e^{-x/n^2}, 
$
where $k\ge 1$ is a real number.  In forthcoming work, they also obtain analogues of this for Dirichlet $L$-functions.  

 In this manuscript, we obtain Riesz-type criteria for the Grand Riemann Hypothesis (GRH) for a broad class of $L$-functions. More precisely, we  obtain criteria which are \textit{sufficient} for GRH and applicable  to functions  $F(s)$ satisfying  axioms of the Selberg class $\mathcal S$ without  imposing  the  Ramanujan hypothesis for  the coefficients of $F(s)$.  On the other hand, we also obtain \textit{necessary} criteria for GRH which apply to functions in a subclass of the Selberg class, which we denote as $\mathcal S^*$.  Essentially, elements of $\mathcal{S^*}$ have a polynomial Euler product and a functional equation involving Gamma functions evaluated on shifts of rational multiples of $s$.  It is expected that this subclass does not exclude any important examples. In fact,  $\mathcal S^*$ is conjectured to be  $\mathcal{S}$.  
 This result is stated as Theorem \ref{Rhlselberg} below. 
 
Some key features of our result are as follows. The number of $\Gamma$-factors of the type $\Gamma(\a s+\b)$ $(\a>0)$ in the functional equation of $F(s)$, for which $\b=0$,  plays an important role in our  results.  This parameter is denoted $j_F$ later in the paper and statements of our results are different depending on whether $j_F$ is zero or non-zero.  Another feature is that we incorporate the degree of the $L$-function  in a new way in order to generalize the  kernel of \eqref{hl criteria}.  It is clear that the natural analogue of $\mu(n)$ is the coefficient appearing in the Dirichlet series of  $1/F(s)$, denoted by $b_F(n)$. However, the kernel $e^{-x/n^2} $ of \eqref{hl criteria} should be thought of as the function $\exp(-y^2)$ evaluated on $\sqrt{x}/ n$. We generalize this to a  function $Z_{\bs \a, \bs \b}\big( (\sqrt{x}/n)^{d_F} \big)$ (see \eqref{defZ} and \eqref{defP}), where $d_F$ is the degree of the $L$-function  $F(s)$ and $\bs \a, \bs \b$ are parameters explicitly arising from the functional equation of $F(s)$. The presence of the degree $d_F$ in the exponent of the argument is a key factor which was not obvious from existing results and plays a crucial role in our theorems.  The function $Z_{\bs \a, \bs \b}(x)$ (see \eqref{defZ}) can essentially (up to a possible residue term)  be viewed  as   an inverse Mellin transform of the product of  Gamma factors arising in the functional equation of $F(s)$. Indeed, this is a natural and apt  generalization if we interpret the function  $\exp({-y^2})$ appearing in \eqref{hl criteria} as the inverse Mellin transform of  the factor $\Gamma(s/2)$ which arises in the functional equation of $\zeta(s)$.  Our methods rely on obtaining  non-trivial estimates (up to a residue term) for  $Z_{\bs \a, \bs \b}(x)$ as well as its derivative  by exploiting  the close connection between  $Z_{\bs \a, \bs \b}(x)$  and the Meijer $G$-function. These are given in Lemmas \ref{lem:Z tilda bd} and \ref{lem:Z tilda dash bd}.

 
  Moreover, along the lines of \cite{riesz},   we introduce an additional parameter $z$ by inserting a $\cosh$ term. This allows for greater flexibility,  enabling us to obtain a sufficient criterion for the case when all but  finitely many zeros of $F(s)$  lie on the critical line. This result (Theorem \ref{finite zero}) is valid for any element $F(s)$ satisfying assumptions of the Selberg class without necessarily satisfying the Ramanujan hypothesis. 
  Our results which  do not require the Ramanujan  hypothesis on the coefficients of the $L$-function  are thus valid for a larger class of $L$-functions, for instance Artin $L$-functions and automorphic $L$-functions (associated with automorphic representations of $\textup{GL}(n)$ over number fields). 
 
The intuition behind  such equivalent criteria for the Riemann Hypothesis lies in  what are known as identities of Ramanujan-Hardy-Littlewood type.  In particular, the criterion \eqref{hl criteria} was motivated via a striking modular transformation  obtained by Ramanujan, involving  infinite series of the M\"obius function \cite{Ramanujan nb} (see also \cite[p.~468, Entry 37]{berndt1}). This identity was later corrected by  Hardy and Littlewood \cite[p.~156, Section 2.5]{hl} to give the  following. 
\begin{theorem}[Ramanujan-Hardy-Littlewood]
Let $\alpha$ and $\beta$ be two positive numbers such that $\alpha\beta=\pi$. Assume that the series $\sum_{\rho}\left(\Gamma{\left(\frac{1-\rho}{2}\right)}/\zeta^{'}(\rho)\right)x^{\rho}$ converges for every positive real $x$, where $\rho$ runs through the non-trivial zeros of $\zeta(s)$, and that the non-trivial zeros of $\zeta(s)$ are simple. Then
\begin{align}\label{mrhl}
	&\sqrt{\alpha}\sum_{n=1}^{\infty}\frac{\mu(n)}{n}e^{-\alpha^2/n^2}-\sqrt{\beta}\sum_{n=1}^{\infty}\frac{\mu(n)}{n}e^{-\beta^2/n^2}=-\frac{1}{2\sqrt{\beta}}\sum_{\rho}\frac{\Gamma{\left(\frac{1-\rho}{2}\right)}}{\zeta^{'}(\rho)}\beta^{\rho}.
\end{align}
\end{theorem}
Such identities can be viewed as encoding relations between arithmetical and analytic pieces of information about the relevant $L$-function. In this case, the left hand side contains arithmetical information about $\zeta(s)$ due to the M\"obius function, whereas the right hand side pertains to analytic information involving  the non-trivial zeros of $\zeta(s)$. 

Showing the convergence of the series on the right-hand side is non-trivial. While it is widely believed that the series is rapidly convergent, what is currently known is that it is convergent under a certain  bracketing of the terms. More precisely, this bracketing means that  for some positive constant $c$,   the terms for which 
\begin{equation} \label{bracketing}
	|\text{Im}\;\rho-\text{Im}\;\rho'|<\exp\left(-c\;\text{Im}\;\rho/\log (\text{Im}\;\rho)\right)+\exp\left(-c\;\text{Im}\;\rho'/\log (\text{Im}\;\rho')\right),
\end{equation}
are included in the same bracket (see \cite[p.~220]{titch}).  It also turns out that convergence under such bracketing is enough to prove \eqref{mrhl} \cite[p.~158]{hl}, \cite[p.~220]{titch}. 

There exist several generalizations of \eqref{mrhl} in various settings.   
 In \cite{hecke}, Dixit, Roy and Zaharescu  obtained such an analogue for Hecke forms.  In \cite{rzz}, Roy, Zaharescu, and Zaki obtained a result of the type  \eqref{mrhl} where the M\"{o}bius function is replaced by a convolution of Dirichlet characters with the M\"{o}bius function. Further results this kind have also been obtained by Dixit \cite{charram, dixthet}, K\"{u}hn, Robles and Roy \cite{krr}, Dixit, Roy and Zaharescu \cite{riesz},  Agarwal, Garg and Maji \cite{agm}, Dixit, Gupta and Vatwani  \cite{dgv}, etc.  
In this paper, we prove a very general such identity for $L$-functions satisfying axioms of the Selberg class other than the Ramanujan hypothesis.   This result is stated as Theorem \ref{rhlselbergth} below. By taking specific examples of  $L$-functions in $\mathcal{S}$, we not only recover many existing such identities in the literature but also obtain elegant new transformations involving special values of  Meijer $G$-functions of the type $G^{n \ 0}_{0 \ n}$, where $n \in \mathbb N$.  Some such consequences are given in Corollaries \ref{coro1} and \ref{coro2}. 

 
 The paper is organised as follows. We first define the Selberg class of  $L$-functions and then proceed to state our main results  in Section 2.  In Section 3, we give some preliminary lemmas which will be needed in our  proofs. In Section 4, we prove Theorem \ref{rhlselbergth}, which is our analogue of the Ramanujan-Hardy-Littlewood identity for the Selberg class. In Section 5, we give proofs of the transformations that arise as special cases of Theorem \ref{rhlselbergth}. Section 6 is devoted to proving Riesz-type criteria for the Grand Riemann Hypothesis for our general $L$-functions.

\section{Statements of results}\label{intro}

	In 1991, Selberg introduced a general class $\mathcal{S}$ of Dirichlet series satisfying certain axioms.
	A Dirichlet series $F(s)$ is defined to be in $\mathcal{S}$ if it satisfies the following conditions. 
	\begin{itemize}
		\item [(i)]  (Dirichlet series)  $F(s)$ can be expressed as a Dirichlet series
		\begin{align}\label{defF}
			F\left(s\right)\mathrm{=}\sum^{\mathrm{\infty }}_{n{=1}}{\frac{a_F\left(n\right)}{n^s}},
		\end{align}
		which is absolutely convergent in the region Re$ (s)> 1$, with $ a_F (1) = 1 $.
		\item[(ii)] (Analytic continuation) There exists a non-negative integer $ m $, 
		such that $(s-1)^m F(s)$ is an entire function of finite order.
		\item[(iii)] (Functional equation) Using the notation $\bar{\Phi}(s)=\overline{\Phi(\bar{s})}$,
		$F(s)$ satisfies a functional equation of the type 
		\begin{equation}\label{Functional equation}
			{\Phi}(s)=\omega \bar{\Phi}(1-s),
		\end{equation}
		where 
		\begin{equation}\label{Functional equation simplified}
			{\Phi}(s)=Q^s F({s}) \prod_{i=1}^q \Gamma\left(\alpha_i s+\beta_i \right) = \gamma_F(s)F(s) \:   (\text{say}), 
		\end{equation}
		where $\ Q>0, \ \alpha_i>0$,  $q\in \mathbb{N}$ and  $\beta_i$, $\omega$ are complex numbers with $\Re (\beta_i)\geq 0$
		and $|\omega|=1$.
		
		\item[(iv)] (Euler product) For $\Re(s)>1$, we have 
		\begin{align}
			F\left(s\right)={\prod_p }F_p\left(s\right),
		\end{align} 
		where 
		$F_p \left(s\right){=}\exp\left( \sum^{{\infty }}_{j{=1}}{\frac{g({p^j})}{p^{js}}}\right) $, 
		with 
		$g({p^j})\ll p^{j\theta}$ 
		for some $\theta < {1\over 2}$.

		\item[(v)] (Ramanujan hypothesis) For every $ \epsilon> 0$,
		\begin{equation}
			a_F\left(n\right)\ll_{\epsilon }n^{\epsilon }. 
		\end{equation}
	\end{itemize}
	
	Many well-known functions such as the Riemann zeta function, Dirichlet $L$-functions, Dedekind zeta functions associated to algebraic number fields etc are elements of the Selberg class.
	The function $\gamma _F(s)$ in axiom (iii) is called the $\gamma$-factor of $F$. This may not be unique, for instance by application of the duplication formula for any of the $\Gamma$-functions involved. However, it is known that the $\g$-factor is unique upto a constant (see Theorem 2.1, \cite{Conrey}). 
	The information in (iii) can be summarized by denoting $(Q,\boldsymbol{\alpha, \beta},w)$ as the data of $F$, where $\bs \a$ and $\bs \b$ are vectors given by  $\boldsymbol{\alpha}=(\alpha_1,\alpha_2,\cdots, \alpha_q)$ and $\boldsymbol{\beta}=(\beta_1,\beta_2,\cdots, \beta_q)$. We will also use  $\boldsymbol{\bar{\beta}}$ to denote $(\bar{\b}_1, \bar{\b}_2\cdots , \bar{\b}_q)$. 
	The data of $F$  thus is not uniquely determined in general. One defines the degree $d_F$ of $F\in \mathcal{S}$ as 
	\begin{align}
		d_F=2\sum_{i=1 }^{q}\alpha_i. 			
	\end{align}
	It is well-known that $d_F$ is an invariant, that is, the degree of $F$ is uniquely determined by the function $F(s)$.
	We note that due to the Euler product (axiom (iv)),  any $F\in \mathcal{S}$ does not vanish in the region $\Re(s)>1$ (for instance see Lemma 2.1 of \cite{Conrey}). Hence there is an arithmetic function $b_F(n)$ such that for $\Re(s)>1$, 
	\begin{align}
		{1\over F(s)}=\sum_{n=1}^{\infty} {b_F(n)\over n^s}.
	\end{align}
	
	In order to understand the analogue of the Riemann Hypothesis for $F\in \mathcal{S}$, some discussion of zeros of $F$ is necessary. 
	The zeros of $F$ in the region $\Re(s)<0$ are called trivial zeros. 
	Using functional equation, it is easy to see that the trivial zeros of $F(s)$ are located at the poles of the $\gamma$-factors $\gamma_F$. More precisely, the trivial zeros occur at 
	\begin{align}
		\left\{s={-m-\beta_i \over \alpha_i}:m\in \mathbb{N}\cup\{0\}\right\}_{i=1}^q
		\cap\{s:\Re(s)<0\}.
	\end{align}
	The other zeros of $F$ occur inside the critical strip $0\leq \Re(s)\leq1$. All such  zeros except at $s=0$ are called the non-trivial zeros of $F$. The case $s=0$ is more delicate. It is often possible for $F$ to have a zero at $s=0$.
	Let 
	\begin{align}\label{defk}
		k_F = \text{order of the pole of $F(s)$ at $s=1$}.
	\end{align} 
If we put $s=0$ into the functional equation of $F$, then the $\gamma$-factors on either side are entire except for poles coming from factors of the type $\Gamma(\a_is+\b_i)$ whenever $\b_i$ happens to be zero. Let $j_F$ be the number of components of $\boldsymbol{\b}=(\b_1,\cdots , \b_q)$ which are zero, that is 
	\begin{align}\label{defj}
		j_F=\#\{1\leq i\leq q:\b_i=0\}.
	\end{align}  
	In order to ensure that $j_F$ is well defined for a given $F\in \mathcal{S}$, we adopt the convention that $\g_F$ is of the form $\prod_{i=1}^q \Gamma\left(\alpha_i s+\beta_i \right)$, with $q$ being the least such integer. That is, if $\g'_F=\prod_{i=1}^{q'} \Gamma\left(\alpha'_i s+\beta'_i \right)$ is any other admissible $\g$-factor for $F$, then we have $q'\geq q.$  Thus, $F(s)$ has a zero of order $j_F-k_F$ at $s=0$. For instance, for the Dedekind zeta function $\zeta_{\mathbb{K}}(s)$, associated to a number field $\mathbb{K}$, one sees that $k_F=1$, while $j_F=r_1+r_2$, where $r_1$ and $2r_2$ are the number of real and imaginary embeddings of $\mathbb{K}$ into $\mathbb{C}$ respectively. Thus $\zeta_{\mathbb{K}}(s)$ has a zero at $s=0$, whenever $r_1+r_2-1>0$. The Grand Riemann Hypothesis (GRH) conjectures that all the non-trivial zeros of $F$ lie on the critical line $\Re(s)={1\over 2}$.  
	
	We now  set up some notation. The line integral $\int_{c-i\infty}^{c+i\infty}$ will be denoted as $\int_{(c)}$. We also define a parameter $c_F$ depending on the data of $F$ as follows: 
	\[ c_F = \left\lbrace 
	\begin{array}{ll}\label{defcF}
		\min \limits_{1\le i\le q}  \left\{  \frac{1}{2\a_i} \right\} 	 	
		&  \text{if } \Re(\beta_i)=0  \ \forall \ 1\leq i\leq q \\ 
		\min \limits_{1\le i\le q}  \left\{  \frac{\Re (\b_i)}{\a_i} \right\} 	 	& \text{otherwise}.
	\end{array} 
	\right. 
	\]	
	It is clear that there are no trivial zeros of $F$ to the right of $\Re(s)=-c_F$.
	Define for $-c_F   <c<0$  and $x>0$, 
	\begin{equation}\label{defZ}
		Z_{ \boldsymbol{\alpha,\beta} }(x) :=\frac{1}{2 \pi i} \int\limits_{(c)} {\prod_{i=1}^q\Gamma\left(\alpha_i s+\bar{\beta_i} \right){x^{ - s}} ds}, 
	\end{equation}	
	and  
	\begin{equation}\label{defP}
		\mathcal{P} _{\boldsymbol{\alpha,\beta} ,z}(y):=\sum\limits_{n = 1}^\infty  \frac{{b_F(n)}}{n}{Z_{\bs{ \a,\b}  }}\left(\left( {\frac{\sqrt{y} }{n}} \right)^{d_F}\right)\cosh\left({\sqrt{y}z}\over n\right), 
	\end{equation}
	for $y>0$ and $z\in \mathbb{C}$. The exponent $d_F$ introduced here in the definition of the kernel $\mathcal{P} _{\boldsymbol{\alpha,\beta} ,z}(y)$ will play a key role in what follows.
	
	Our first result is a modular relation for functions $F$ in the Selberg class. We will denote the derivative of $F(s)$ with respect to $s$  by $F'(s)$. 
	\begin{theorem}\label{rhlselbergth}
	Let $F$ satisfy axioms (i)-(iv) of the class $\mathcal{S}$. For every positive real $x$, assume the convergence of the series 
		\begin{equation}
			\sum\limits_{\rho} \frac{\prod_{i=1}^q\Gamma\left(\alpha_i(1- \bar{\rho})+\beta_i \right)}{\overline{F'({\rho})}}x^{\bar{\rho}},
		\end{equation} 
		where $\rho$ runs through the non-trivial zeros of $F(s)$. Suppose that each non-trivial zero of $F(s)$ is simple. Let $\eta , \nu >0$ be such that $\eta \nu ={1\over Q^2}$.  Let $r=k_F-j_F$, where $k_F, j_F$ are as defined in \eqref{defk}, \eqref{defj} respectively. Then for $r>0$, we have 
		\begin{align}
			\omega\sqrt{\eta}&	\sum_{n=1}^{\infty} {b_{F}(n)\over n}Z_{\boldsymbol{\alpha,\beta} }\left(\eta \over n\right)-{\sqrt{\nu} }\sum_{n=1}^{\infty} {\overline{b_{F}(n)}\over n}Z_{\boldsymbol{\alpha,\overline{\beta}}}\left(\nu \over n\right)\\
			&= -{1\over \sqrt{\nu} }\sum_{\rho} \frac{\prod_{i=1}^q\Gamma\left(\alpha_i(1-\bar{\rho}) +\beta_i \right)}{\overline{F'({\rho})}}\nu^{\bar{\rho}} 
			-	{1\over \sqrt{\nu}(r-1)!}{d^{r-1}\over ds^{r-1}}(s-1)^r   \frac{\prod_{i=1}^q\Gamma\left(\alpha_i(1-s)+\beta_i \right)}{\overline{F(\bar{s})}}\nu^{s}\bigg|_{s=1}\\
			& \quad -	{1\over \sqrt{\nu}(r-1)!}{d^{r-1}\over ds^{r-1}}s^r   \frac{\prod_{i=1}^q\Gamma\left(\alpha_i(1-s)+\beta_i \right)}{\overline{F(\bar{s})}}\nu^{s}\bigg|_{s=0} \label{rhlselbergeq}.
		\end{align} 
		If $r\leq 0$, we have the following more simplified relation:
		\begin{align}\label{rhlselbergeqr0}
			\omega\sqrt{\eta}	\sum_{n=1}^{\infty} {b_{F}(n)\over n}Z_{\boldsymbol{\alpha,\beta} }\left(\eta \over n\right)-{\sqrt{\nu} }\sum_{n=1}^{\infty} {\overline{b_{F}(n)}\over n}Z_{\boldsymbol{\alpha,\overline{\beta}}}\left(\nu \over n\right) = -{1\over \sqrt{\nu} }\sum_{\rho} \frac{\prod_{i=1}^q\Gamma\left(\alpha_i(1-\bar{\rho}) +\beta_i \right)}{\overline{F'({\rho})}}\nu^{\bar{\rho}} .
		\end{align}
	\end{theorem}
	
	A straightforward consequence of Theorem 1.1 is that the classical result of Ramanujan, Hardy and Littlewood follows as a special case since for $F(s)=\zeta(s)$, we have $Z_{\boldsymbol{\a,\b}}(x) =Z_{(1/2),(0)}(x)$, which equals $2(e^{-x^2}-1)$. 
	 Moreover when $F$ is the Dedekind zeta function of a number field $K$, it is easy to check that 
	 \begin{align}
\boldsymbol{\a}=(\underbrace{\substack{{1\o 2},\cdots,{1\o 2}}}_\text{$r_1$ times},\underbrace{1,\cdots, 1}_\text{$r_2$ times}), \quad \boldsymbol{\b}=(0,\cdots, 0) \quad \text{and} \quad  Z_{\boldsymbol{\a,\b}}(x)=\frac{1}{2 \pi i} \int_{(c)} {{\Gamma ^{{r_1}}}\left( {\frac{s}{2}} \right){\Gamma ^{{r_2}}}\left( s \right){x^{ - s}}ds}, 
	 \end{align}
	  where $r_1$ and $2r_2$ are the number of real and complex embeddings of $K$ respectively. It thus follows that a special case of Theorem 1.1 yields Theorem 1.1 and Corollary 3.2 of \cite{dgv}.
	 
	 We now state more non-trivial consequences of Theorem 1.1 among which are certain elegant transformations involving special values of the Meijer $G$-function (defined in \eqref{meijer G}). In particular, special values of Meijer $G$-functions of the type $G^{n \ 0}_{0 \ n}$ for $n=1,2,3,\cdots$ come into play. Some interesting such special values involve the modified Bessel function of the second kind, which we proceed to define below. 
	 
	 The Bessel function of the first kind of order $\nu$ is defined by \cite[p.~40]{watson}
	 \begin{align}
	 	J_{\nu}(z)&:=\sum_{m=0}^{\infty}\frac{(-1)^m(z/2)^{2m+\nu}}{m!\Gamma(m+1+\nu)} \hspace{9mm} (z,\nu\in\mathbb{C}). 
	 \end{align}
	 The modified Bessel functions of the first and second kinds of order $\nu$ are defined by \cite[pp.~77-78]{watson}
	 \begin{align}
	 	I_{\nu}(z)&:=
	 	\begin{cases}
	 		e^{-\frac{1}{2}\pi\nu i}J_{\nu}(e^{\frac{1}{2}\pi i}z), & \text{if $-\pi<$ arg $z\leq\frac{\pi}{2}$,}\\
	 		e^{\frac{3}{2}\pi\nu i}J_{\nu}(e^{-\frac{3}{2}\pi i}z), & \text{if $\frac{\pi}{2}<$ arg $z\leq \pi$,}
	 	\end{cases}\\
	 	K_{\nu}(z)&:=\frac{\pi}{2}\frac{I_{-\nu}(z)-I_{\nu}(z)}{\sin\nu\pi}
	 \end{align}
	 respectively, with $K_n(z)$ defined by $\lim_{\nu\to n}K_{\nu}(z)$ if $n$ is an integer.
	 
	We state some special values of the Meijer G-function which are relevant in the context of Theorem \ref{rhlselbergth}.
	
		\begin{align}
		G^{4 \  0}_{0 \ 4}\left( \left. \begin{matrix}
			\rule{.4cm}{.2mm} \\ b \ b+{1\over 4} \  b+{2\over 4}\ b+{3\over 4}
		\end{matrix} \right| z \right)
	&={\sqrt{2}\pi^{3\over 2}}z^be^{-4 {z}^{1/4}},\label{sepG1}
		\\
		G^{4 \  0}_{0 \ 4}\left( \left. \begin{matrix}
			\rule{.4cm}{.2mm} \\ b \ b+{1\over 2} \  2b-c\ c
		\end{matrix} \right| z \right)
	&=8{\sqrt{\pi}}z^b K_{2c-2b}(2\sqrt{2} (-z)^{1/4} )K_{2c-2b}\left({2\sqrt{2}\sqrt{z}\o (-z)^{1/4}}\right),\label{sepG2}
	\end{align}  
	
	Using the identity \eqref{sepG1}, we derive the following transformation.

		\begin{corollary}\label{coro1}
		Let $F\in \mathcal{S}$ with data $(Q, \boldsymbol{\alpha,\beta}, \omega)$, where $\boldsymbol{\alpha}=(4)$ and $\boldsymbol{\beta}=(\b_1)$. For every positive real $x$, assume the convergence of the series 
		\begin{equation}
			\sum_{\rho} \frac{\Gamma\left(4(1-\bar{\rho}) + \bar \b_1 \right)}{\overline{F'({\rho})}}x^{\bar{\rho}},
		\end{equation} 
		where $\rho$ runs through the non-trivial zeros of $F(s)$. Suppose that each non-trivial zero of $F(s)$ is simple. Let $\eta , \nu >0$ be such that $\eta \nu ={1\over Q^2}$.  Let $r=k_F-1$, where $k_F$ is as defined in \eqref{defk}. Then for $\b_1=0$, we have
		\begin{align}
			{\omega\sqrt{\eta}\o 4}&	\sum_{n=1}^{\infty} {b_{F}(n)\over n}\left(e^{-\left( {\eta \over n}\right)^{1/4}  }-{1}\right)-{\sqrt{\nu} \o 4}\sum_{n=1}^{\infty} {\overline{b_{F}(n)}\over n}\left(e^{-\left( {\nu \over n}\right)^{1/4}  }-{1}\right)\\
			&= -{1\over \sqrt{\nu} }\sum_{\rho} \frac{\Gamma\left(4(1-\bar{\rho})  \right)}{\overline{F'({\rho})}}\nu^{\bar{\rho}} 
			-	{1\over \sqrt{\nu}(r-1)!}{d^{r-1}\over ds^{r-1}}(s-1)^r   \frac{\Gamma\left(4(1-s) \right)}{\overline{F(\bar{s})}}\nu^{s}\bigg|_{s=1}\\
			& \quad -	{1\over \sqrt{\nu}(r-1)!}{d^{r-1}\over ds^{r-1}}s^r   \frac{\Gamma\left(4(1-s)\right)}{\overline{F(\bar{s})}}\nu^{s}\bigg|_{s=0}\label{coro1eqn} .
		\end{align}
		When $\b_1\neq 0$, we have 
		\begin{align}
			{\omega\sqrt{\eta}\o 4}&	\sum_{n=1}^{\infty} {b_{F}(n)\over n}	\left(\eta \over n \right)^{\bar{\b}_1/ 4}e^{-\left( {\eta \over n}\right)^{1/4}  }-{\sqrt{\nu} \o 4}\sum_{n=1}^{\infty} {\overline{b_{F}(n)}\over n}	\left(\nu \over n \right)^{\bar{\b}_1/ 4}e^{-\left( {\nu \over n}\right)^{1/4}  }\\
			&= -{1\over \sqrt{\nu} }\sum_{\rho} \frac{\Gamma\left(4(1-\bar{\rho})+\bar{\beta_1}  \right)}{\overline{F'({\rho})}}\nu^{\bar{\rho}} 
			-	{1\over \sqrt{\nu}(r-1)!}{d^{r-1}\over ds^{r-1}}(s-1)^r   \frac{\Gamma\left(4(1-s) +\bar{\beta_1}\right)}{\overline{F(\bar{s})}}\nu^{s}\bigg|_{s=1}\\
			& \quad -	{1\over \sqrt{\nu}(r-1)!}{d^{r-1}\over ds^{r-1}}s^r   \frac{\Gamma\left(4(1-s)+\bar{\beta_1}\right)}{\overline{F(\bar{s})}}\nu^{s}\bigg|_{s=0}\label{coro2eqn} .
		\end{align}
	\end{corollary}

The identity \eqref{sepG2} yields the following transformation involving the modified Bessel function of the second kind.

	\begin{corollary}\label{coro2}
	Let $F\in \mathcal{S}$ with data $(Q,(1,1,1,1),(0,{1\o 2},-C,C),\omega)$, where $C$ is a non-zero real constant. For every positive real $x$, assume the convergence of the series 
	\begin{equation}
		\sum_{\rho} \frac{\Gamma(1-\bar{\rho})  \Gamma\left({3\o 2}-\bar{\rho} \right) \Gamma(1-C-\bar{\rho})  \Gamma(C+1-\bar{\rho})  }{\overline{F'({\rho})}}x^{\bar{\rho}},
	\end{equation} 
	where $\rho$ runs through the non-trivial zeros of $F(s)$. Suppose that each non-trivial zero of $F(s)$ is simple. Let $\eta , \nu >0$ be such that $\eta \nu ={1\over Q^2}$.  Let $r=k_F-1$, where $k_F$ is as defined in \eqref{defk}. Then we have
	\begin{align}
		\omega\sqrt{\eta}&	\sum_{n=1}^{\infty} {b_{F}(n)\over n}\left(8{\sqrt{\pi}} K_{2C}\left(2\sqrt{2}  \left( {-{\eta\o n}} \right)^{1/4} \right)K_{2C}\left({2\sqrt{2}\sqrt{{\eta\o n}}\o  \left( {-{\eta\o n}} \right)^{1/4}  }\right)+{\pi\sqrt{\pi}\o C\sin \pi C}\right)\\
		&\quad-{\sqrt{\nu} }\sum_{n=1}^{\infty} {\overline{b_{F}(n)}\over n}\left(8{\sqrt{\pi}} K_{2C}\left(2\sqrt{2} \left( {-{\nu\o n}} \right)^{1/4} \right)K_{2C}\left({2\sqrt{2}\sqrt{{\nu\o n}}\o \left( {-{\nu\o n}} \right)^{1/4}    }\right)+{\pi\sqrt{\pi}\o C\sin \pi C}\right)\\
		&= -{1\over \sqrt{\nu} }\sum_{\rho} \frac{\Gamma(1-\bar{\rho})  \Gamma\left({3\o 2}-\bar{\rho} \right) \Gamma(1-C-\bar{\rho})  \Gamma(C+1-\bar{\rho})  }{\overline{F'({\rho})}}\nu^{\bar{\rho}} \\
		&
		\quad-	{1\over \sqrt{\nu}(r-1)!}{d^{r-1}\over ds^{r-1}}(s-1)^r   \frac{\Gamma(1-s)  \Gamma\left({3\o 2}-s \right) \Gamma(1-C-s)  \Gamma(C+1-s)}{\overline{F(\bar{s})}}\nu^{s}\bigg|_{s=1}\\
		& \quad -	{1\over \sqrt{\nu}(r-1)!}{d^{r-1}\over ds^{r-1}}s^r   \frac{\Gamma(1-s)  \Gamma\left({3\o 2}-s \right) \Gamma(1-C-s)  \Gamma(C+1-s)}{\overline{F(\bar{s})}}\nu^{s}\bigg|_{s=0}\label{coro4eqn} .
	\end{align}
\end{corollary}
Similarly, the identities \eqref{sepG01}  given below can be used to recover Corollaries 3.4 and 3.3 of \cite{dgv} respectively. 
\begin{align}
		G^{1 \  0}_{0 \ 1}\left( \left. \begin{matrix}
	\rule{.4cm}{.2mm} \\ b 
\end{matrix} \right| z \right)
=e^{-z}z^b, 
\qquad  \qquad
G^{2 \  0}_{0 \ 2}\left( \left. \begin{matrix}
	\rule{.4cm}{.2mm} \\ b \ b+{1\over 2}
\end{matrix} \right| z \right)
=\sqrt{\pi}z^be^{-2\sqrt{z}}. 
\label{sepG01}
\end{align}
 It is also possible to obtain more such transformations involving special values of $G^{n \ 0}_{0 \ n} $ for  other values of $n$, for instance using any of the following  identities: 
\begin{align}
		&G^{2 \  0}_{0 \ 2}\left( \left. \begin{matrix}
	\rule{.4cm}{.2mm} \\ b \ c 
\end{matrix} \right| z \right)
=2z^{{1\over 2}(b+c)}K_{b-c}(2\sqrt{z}), \label{sepG03} 
\qquad 
G^{3 \  0}_{0 \ 3}\left( \left. \begin{matrix}
	\rule{.4cm}{.2mm} \\ b \ b+{1\over 3} \  b+{2\over 3}
\end{matrix} \right| z \right)
={2\pi\over \sqrt{3}}z^be^{-3 {z}^{1/ 3}}. \label{sepG04}
\\
&G^{5 \  0}_{0 \ 5}\left( \left. \begin{matrix}
	\rule{.4cm}{.2mm} \\ b \ b+{1\over 5} \  b+{2\over 5}\ b+{3\over 5}\ b+{4\over 5}
\end{matrix} \right| z \right)
={4\pi^{ 2}\over\sqrt{5}}z^be^{-5 {z}^{1/ 5}}\label{sepG05}.
\end{align}
Since the proofs follow analogously to those of Corollaries  \ref{coro1} and \ref{coro2}, we do not delve into them here. 

	Our next main result gives Riesz-type criteria for the Grand Riemann Hypothesis, applicable to a general class of $L$-functions. Before stating this, we make some remarks about certain additional conditions that we need to impose on our $L$-function $F(s)$.
	
	The first such condition pertains to (iv) of the Selberg class axioms. From the perspective of known examples, in particular automorphic $L$-functions, it is natural to restrict $1\o F_p (s)$ to be a polynomial in $p^{-s}$ of degree independent of $p$ (see for instance \cite{Conrey}).  
	Axiom (iv) may thus be strengthened to axiom (iv)$'$ given below. We say that $F$ has polynomial Euler product if it satisfies the following 
	
	(iv)$'$ (Polynomial Euler product) For $\Re(s)>1$, we have 
	 \begin{align}
	 	F(s)=\prod_p\prod_{i=1}^{n}\left( 1-{\gamma_i(p)\o p^s}\right)^{-1} ,
	 \end{align}
	where $\gamma_i(p)$ are complex numbers.  It is conjectured that all $L$-functions in the Selberg class have a polynomial Euler product. 
		It is well known that under (iv)$'$, the axiom (v) is equivalent to the bound 
	\begin{align}\label{inequality}
		|\gamma_i(p)|\leq 1,
	\end{align}
	for all primes $p$ and $i=1,\cdots ,m$ (see for instance p. 347 of \cite{Christoph}).
	
	In all known examples of $L$-functions from $\mathcal{S}$, it is possible to find a normalization such that in axiom (iii), one has all $\a_i$'s to be rational numbers. In fact Conrey and Ghosh observed that one could take $\a_i={1\o 2}$ for all $i$ (cf. p.12 of \cite{Conrey}). We thus formulate a  subclass of functions called $\mathcal{S}^*$ 
	 satisfying axioms (i)-(iii),(iv)$'$,(v), with the additional restriction that in the data $(Q,\boldsymbol{\a,\b},w)$ of $F$, all the components of $\boldsymbol{\a}$ are rational numbers. 
	 
	 We do not appear to exclude any known examples of important $L$-functions under these new hypotheses. Indeed, it is conjectured that $\mathcal{S}^*=\mathcal{S}$. Our next result gives Riesz-type criteria for the Grand Riemann Hypothesis (GRH) in this setting.

	\begin{theorem}\label{Rhlselberg}
		Let $F$ satisfy axioms (i)-(iv) of the class $\mathcal{S}$. Let $j_F$ be as defined in \eqref{defj}. Then  the following hold. 
			\begin{itemize}
				\item [(a)] If $F\in \mathcal{S^*}$ and $j_F=0$, then the \textup{GRH} for $F(s)$ implies that $\mathcal{P} _{\boldsymbol{\alpha,\beta} ,z}(y)={O}_{F,\d, z}\left( y^{-\frac{1}{4}+\delta} \right)$ as $y\to \i$, for any $\d>0$.
				\item [(b)]  Let $\e>0$. If $F\in \mathcal{S^*}$ and $j_F\neq 0$, then the \textup{GRH} for $F(s)$ implies that 
				\begin{align}
					\mathcal{P} _{\boldsymbol{\alpha,\beta} ,0}(y)+	\sum_{n=1}^{[y ^{{1\over 2}-\epsilon}] -1}{b_F(n) \over n} {\Res\limits_{s=0} } \prod_{i=1}^q\Gamma\left(\alpha_i s+\bar{\beta_i} \right)\left({\sqrt{y}\o n} \right)^{-d_Fs} \ll_{F,\d} y^{-\frac{1}{4}+\delta} \label{est 1.4b}
				\end{align} 
			as $y\to \i$ for any $\d>0$.
				\item [(c)] The estimate $\mathcal{P} _{\boldsymbol{\alpha,\beta} ,0}(y)={O}_{F,\d}\left( y^{-\frac{1}{4}+\delta} \right)$ as $y\to \i$ for any $\d>0$ implies the \textup{GRH}  for $F(s).$ 
			\end{itemize}
		\end{theorem} 
	Part $(c)$ is valid for the larger class of functions satisfying axioms (i)-(iv) of the Selberg class.  
We remark that our result  gives an equivalent condition for the GRH when $F\in \mathcal{S^*}$ satisfies $j_F=0$, namely we show that in this case GRH is equivalent to the estimate $\mathcal{P} _{\boldsymbol{\alpha,\beta} ,0}(y)\ll_{F,\d}\left( y^{-\frac{1}{4}+\delta} \right)$. The role of $j_F$ (and hence the data of $F$) is a new feature of our result, and has not been recorded explicitly before. 
	
	Until this point, we have not made use of the flexibility offered by the parameter $z$ in the definition \eqref{defP} of $\mathcal{P} _{\boldsymbol{\alpha,\beta} ,z}(y)$. This parameter allows us to give a criterion for an `almost' GRH phenomenon as follows. This is a generalization of Theorem 1.1(b) of \cite{riesz}.
	\begin{theorem}\label{finite zero} 
			Let $F$ satisfy axioms (i)-(iv) of the class $\mathcal{S}$. 
		If $z \ne 0$ and $\arg z \ne \pm {\pi \o 2}$, then 
		the estimate $\mathcal{P} _{\boldsymbol{\alpha,\beta} ,z}(y)={O}_{F, \d}\left( y^{-\frac{1}{4}+\delta} \right)$,  as $y\to \i$ for any $\d>0$ implies that $F(s)$ has at most finitely many non-trivial zeros off the critical line.
	\end{theorem}
	  
\begin{remark*}
	We note that the Ramanujan bound (axiom (v)) for $F(s)$ will be used only  in the proofs of parts (a) and (b) of Theorem \ref{Rhlselberg}. In particular, it is used to conclude that GRH implies square root cancellation in the summatory function of $b_F(n)$, which is in turn used to derive  the estimate \eqref{bd on MFhn}.   
\end{remark*}

	\section{Preliminary results}\label{prelim}
Throughout this section, let $F\in \mathcal{S}$. Let us consider 
	\begin{align}\label{defNF}
		N_F(T) : = \# \{s: F(s)=0 ,0 < \Re(s) <1 \text{ and } |\Im (s)| <T \}.
	\end{align}
	It is well-known that as $T\to \infty $, $N_F(T)$ satisfies the asymptotic formula (see \cite[p. 264]{Anirban})
	\begin{equation}
		N_F(T)= \frac{d_F}{\pi }T{\log}T+CT+O(\log T),
	\end{equation}
	where $d_F$ is the degree of $F$ and $C=C(F)$ is some constant. This in particular gives 
	\begin{equation}\label{dens}
		N_F\left(T+1\right)-N_F\left(T\right)\ll \log T.
	\end{equation} 
	From \eqref{dens}, one can conclude that the  number of zeros of $F(s)$ in the critical strip between the horizontal lines Im$(s)=T+1$ and Im$(s)=T-1$ is ${O}(\log T)$.

Another result that will be used repeatedly in our proofs is the following well-known estimate for the Gamma function. 
	\begin{lemma}[Stirling's formula,\cite{cop}]\label{lemmastir}
		For $s=\sigma+it$ with $C\leq\sigma\leq D$, we have as $|t| \to \infty$, 
		\begin{equation}\label{strivert}
			|\Gamma(s)|=(2\pi)^{\frac{1}{2}}|t|^{\sigma-\tfrac{1}{2}}e^{-\frac{1}{2}\pi |t|}\left(1+O\left(\frac{1}{|t|}\right)\right).
		\end{equation}
	\end{lemma}
	The Meijer G-function is defined as \cite[p. 143]{luke}
	\begin{align}\label{meijer G}
		G_{p,q}^{m,n}\left( \left. \begin{matrix}
			a_1 & \cdots & a_p\\ b_1 & \cdots &  b_q
		\end{matrix} \right| z \right)=\frac{1}{2 \pi i}\int\limits_L\frac{\Pi_{j=1}^m \Gamma(b_j-s) \Pi_{j=1}^n \Gamma(1-a_j+s)}{\Pi_{j=m+1}^{q} \Gamma(1-b_j+s) \Pi_{j=n+1}^{p} \Gamma(a_j-s)} z^sds,
 	\end{align}
	where $L$ is a curve from $-i\i$ to $i\i$ which separates the poles of the factors $\Gamma(b_j-s)$ from those of the factors $\Gamma(1-a_j+s)$. For more details, we refer the reader to \cite[p. 143]{luke}.
	
Recall the definitions \eqref{defZ} and \eqref{defP} of ${Z_{\boldsymbol{\a,\b}}(x)}$ and $\mathcal{P} _{\boldsymbol{\alpha,\beta} ,z}(y)$ respectively. We will prove some preliminary results about these functions.
\begin{lemma}\label{lemmaz0}
Let $0 <\textup{Re}(s)<{1 \over 2}$, $\boldsymbol{\a}=(\a_1,\cdots, \a_q), \ \boldsymbol{\b}=(\b_1,\cdots, \b_q),$ with  $\a_i>0$ and $\Re(\b_i)>0$ for each $i$. Then
\begin{equation}\label{lemma2}
\int_{0}^{\infty}y^{-s-1}\mathcal{P} _{\boldsymbol{\alpha,\beta} ,z}(y)dy=\frac{2 }{d_F F(2s+1)}\sum_{t=0}^{\i}{z^{2t}\over(2t)!}  {\prod_{i=1}^q\Gamma\left({2\alpha_i\o d_F} (t-s)+\bar{\beta_i} \right)}.
		\end{equation}
	\end{lemma}
	
	\begin{proof}
Let \begin{align}
\phi (s,\boldsymbol{\alpha,\beta} ,z)=\int_{0}^{\infty}y^{-s-1}\mathcal{P} _{\boldsymbol{\alpha,\beta} ,z}(y)dy.
		\end{align}
Substituting $y$ by ${x\o n^2}$, we have 
		\begin{align}
n^{-2s-1}\phi (s,\boldsymbol{\alpha,\beta} ,z)=\int_{0}^{\infty}{x^{-s-1}\o n}\mathcal{P} _{\boldsymbol{\alpha,\beta} ,z}\left( {x\o n^2}\right) dx.
		\end{align}
Multiplying both sides by $a_F(n)$ and summing over all $n$, we obtain from \eqref{defP}
		\begin{align}
F&(2s+1)\phi (s,\boldsymbol{\alpha,\beta} ,z)\\
			&=\sum\limits_{n = 1}^\infty\int_{0}^{\infty}{x^{-s-1}a_F(n)\o n}\sum\limits_{m = 1}^\infty  \frac{{b_F(m)}}{m}{Z_{\a,\b}}\left( \left({\frac{\sqrt{x} }{mn}}\right)^{d_F} \right) \cosh\left({\sqrt{x}z\over nm}\right) dx\nonumber \\
			&=\sum\limits_{n = 1}^\infty\int_{0}^{\infty}{x^{-s-1}a_F(n)\o n}\sum\limits_{m = 1}^\infty  \frac{{b_F(m)}}{m}\sum_{t=0}^{\i}{z^{2t}x^t\over m^{2t}n^{2t}(2t)!} \frac{1}{2 \pi i} \int\limits_{(c)} \prod_{i=1}^q\Gamma\left(\alpha_i s'+\bar{\beta_i} \right){\left( {\frac{\sqrt{x} }{mn}} \right)^{ -d_F s'}}ds' dx,
		\end{align}	
		using the series expansion for $\cosh(x)$. Applying  the Weierstrass M-test and the Lebesgue dominated convergence theorem, one can show that 
		\begin{align}
			F(2s+1)&\phi (s,\boldsymbol{\alpha,\beta} ,z)\\
			&=\sum_{t=0}^{\i}{z^{2t}\over(2t)!} \int_{0}^{\infty}{x^{-s+t-1}}\frac{1}{2 \pi i} \int\limits_{(c)} \frac{\prod_{i=1}^q\Gamma\left(\alpha_i s'+\bar{\beta_i} \right)}{F(2t+1-d_Fs')}\sum\limits_{n = 1}^\infty{a_F(n)\o n^{2t+1-d_Fs'}}{\left( {{\sqrt{x} }} \right)^{ - d_Fs'}}ds' dx\nonumber \\
			&=\sum_{t=0}^{\i}{z^{2t}\over(2t)!}\int_{0}^{\infty}{x^{-s+t-1}}\frac{1}{2 \pi i} \int\limits_{(c)} {\prod_{i=1}^q\Gamma\left(\alpha_i s'+\bar{\beta_i} \right)}{\left( {{\sqrt{x} }} \right)^{ - d_Fs'}}ds' dx.\label{eq411}
		\end{align} 
Substituting  $s'$ by ${-2w\o d_F}$, we have
		\begin{align}
\frac{1}{2 \pi i}	\!	\int_{0}^{\infty} \!\!\! {x^{-s+t-1}}  \!\! \int\limits_{(c)} \! {\prod_{i=1}^q\Gamma\left(\alpha_i s'+\bar{\beta_i} \right)}{{{{x} }} ^{ - d_Fs' \o 2}}ds'  dx
&=\frac{2}{d_F2 \pi i} \int_{0}^{\infty}\!\!\!\! {x^{-s+t-1}} \!\!\!\! \int\limits_{{\left( -cd_F\o 2\right) }} \! {\prod_{i=1}^q\Gamma\left(- \tfrac{2\alpha_i w} { d_F}+\bar{\beta_i} \right)}{x} ^{w}dw  dx\\
&= {2\o d_F}{\prod_{i=1}^q\Gamma\left({2\alpha_i\o d_F} (t-s)+\bar{\beta_i} \right)},\label{eq412}
		\end{align}
upon replacing $x$ by $\frac{1}{x}$ and using the Mellin inversion theorem \cite[p. 341-343]{Mcla}. Combining \eqref{eq411} and \eqref{eq412} completes the proof.
\end{proof}
	

	\begin{lemma} \label{lem:Z poly bound}
		For $-c_F<c<0$ and $x>0$, we have 
		\begin{align}
			Z_{\boldsymbol{\alpha,\beta} }(x)\ll_F x^{-c}.
		\end{align}
	\end{lemma}	
	\begin{proof}
		Let us write  $\bar \b_i=a_i+ib_i$. Using the functional equation of $\G$, we have from \eqref{defZ},
		\begin{align}               
			Z_{\boldsymbol{\alpha,\beta} }(x)
			& =\frac{1}{2 \pi i} \int_{c-i\infty}^{c+i\infty} {\prod_{i=1}^q\Gamma\left(\alpha_i s+\bar{\beta_i}+1 \right)\o\prod_{i=1}^q\left(\alpha_i s+\bar{\beta_i} \right) }{x^{ - s}}ds\\
			& =\frac{1}{2 \pi i} \int_{-\infty}^{\infty} {\prod_{i=1}^q\Gamma\left(\alpha_ic +a_i +1+i(\a_it-b_i) \right)\o\prod_{i=1}^q\left(\alpha_ic +a_i +i(\a_it-b_i) \right) }{x^{ -c-it}}dt.
		\end{align}	
		It can be seen that 
		\begin{align}
			1+|t|&\ll_F|\alpha_ic +a_i| +|\a_it-b_i|\\
			&\ll_{F} |\alpha_ic +a_i +i(\a_it-b_i)|, 
		\end{align}
		using the Cauchy-Schwarz inequality. Using this for the denominator and  Stirling's formula for the numerator, we get 
		\begin{align}
			Z_{\boldsymbol{\alpha,\beta} }(x)&\ll_F \int_{-\i}^{\i}\frac{\prod_{i=1}^q\left( |\a_it-b_i|^{\a_ic+a_i+{1\o 2}}e^{-{\pi\o 2}|\a_it-b_i|}\right) }{(1+|t|)^{q}}x^{-c}dt.
		\end{align}
		Noting that the exponent $\a_ic+a_i+{1\o2}$ is positive in the region $-c_F<c<0$ due to the choice of $c_F$ (see \eqref{defcF}), it is easy to see that 
		$
		Z_{\boldsymbol{\alpha,\beta} }(x)\ll_Fx^{-c}, 
		$ as needed. 
	\end{proof}
	
	\begin{lemma} \label{lem:Z' bd}
		For $-c_F<c<0$ and $x>0$, we have 
		\begin{align}
			Z'_{\boldsymbol{\alpha,\beta} }\left( { x}\right) \ll_{F}x^{-c-1}, 
		\end{align}
		where $Z'_{\boldsymbol{\alpha,\beta} }(x)$ denotes the derivative of $ Z_{\boldsymbol{\alpha,\beta} }\left( { x}\right)$ with respect to its argument $x$. 
	\end{lemma}
	\begin{proof}
		The proof follows in exactly the same manner as that of Lemma  \ref{lem:Z poly bound}. 
	\end{proof}

For $x\in \mathbb{R} $ and $1<d<1+c_F$, we define the following function which is closely related to the function ${Z}_{\boldsymbol{\alpha,\beta} }(x)$ defined in \eqref{defZ}. Let 
\begin{equation}\label{def Z tilde}
	\tilde{Z}_{{\boldsymbol{\alpha,\beta}}}(x) :=\frac{1}{2 \pi i} \int\limits_{(d)} {\prod_{i=1}^q\Gamma\left(\alpha_i s+\bar{\beta_i} \right){x^{ - s}}ds}.
\end{equation}
The advantage of using $\tilde{Z}_{{\boldsymbol{\alpha,\beta}}}(x)$ is that one is able to obtain better upper bounds for this function by relating it to the Meijer $G$-function. Moreover, $\tilde{Z}_{{\boldsymbol{\alpha,\beta}}}(x)$ differs from ${Z}_{{\boldsymbol{\alpha,\beta}}}(x)$ only by a residue term. More precisely,  using Cauchy's residue theorem, we have 
	\begin{align} \label{Z vs Z tilda}
	{Z}_{\boldsymbol{\alpha,\beta} }\left(x \right) =\tilde{Z}_{\boldsymbol{\alpha,\beta} }\left(x \right)-\Res_{s=0} \prod_{i=1}^{q}\Gamma\left(\a_is+\bar{\beta_i} \right){\left(x \right)^{ -s}}. 
\end{align}
In what follows, we   show  exponential decay  as $x\to \infty$ for both $\tilde{Z}_{\boldsymbol{\alpha,\beta} }\left( { x}\right)$ as well as  $\tilde{Z}'_{\boldsymbol{\alpha,\beta} }\left( { x}\right)$  (its derivative with respect to $x$). These results are used  when $\tilde{Z}_{\boldsymbol{\alpha,\beta}}$ arises from the data of a function $F$ in the class $\mathcal{S}^*$ defined preceding Theorem \ref{Rhlselberg}. . 
	
	\begin{lemma} \label{lem:Z tilda bd}
		Let $\boldsymbol{\a}=(\a_1,\cdots, \a_q), \ \boldsymbol{\b}=(\b_1,\cdots, \b_q),$ with  $\a_i>0$ and $\Re(\b_i)>0$ for all $i$. Assume that each component of  $\boldsymbol{\a}$ is rational, say $\a_i={m_i\o n_i}$ with $(m_i, n_i)=1$. Let $L$ be the lcm $[n_1, \cdots , n_q]$. Denote the positive integer $\a_i L$ by $k_i$. Using the notation
		\begin{align}\label{Ci}
			C_1=\frac{ L d_F}{ 2\left({ \prod_{i=1}^qk_i^{k_i} }\right)^{2/ Ld_F}}, \quad C_2= \frac{2}{d_F} ,\quad C_3=\sum_{i=1}^{q}\bar{\b_i}+{1-q\o 2},
		\end{align} 
		we have as $x\to \infty$,  
		\begin{align}
			\tilde{Z}_{\boldsymbol{\alpha,\beta} }(x)&\ll_{F}  \exp\left(-C_1x^{C_2}\right)x^{C_2C_3}.
		\end{align}
	\end{lemma}
	
	\begin{proof}
		We have from the definition 
		\begin{align}
			\tilde{Z}_{\boldsymbol{\alpha,\beta} }(x)&=\frac{1}{2 \pi i} \int\limits_{(d)} {\prod_{i=1}^q\Gamma\left(\alpha_i s+\bar{\beta_i} \right){x^{ - s}}ds},
		\end{align}
		where $d>0$. The change of variable $s=Lz$ gives 
		\begin{align}
			\tilde{Z}_{\boldsymbol{\alpha,\beta} }(x)&=\frac{L}{2 \pi i} \int\limits_{\left( d/L\right) } {\prod_{i=1}^q\Gamma\left(k_iz+\bar{\beta_i} \right){x^{ -z L}}dz}. 
		\end{align}
		Using the Gauss multiplication formula, we have 
		\begin{align}
			\Gamma\left(k_i\left(z+{\bar{\beta_i}\o k_i} \right)\right)=(2\pi)^{1-k_i \o 2}k_i^{k_i\left(z+{\bar{\beta_i}\o k_i} \right)-{1\o 2}}\prod_{j=1}^{k_i}\Gamma\left(z+{\bar{\beta_i}\o k_i} +{j-1\o k_i}\right).
		\end{align}
		Thus, $\tilde{Z}_{\boldsymbol{\alpha,\beta} }(x)$ equals
		\begin{align}
			\frac{L}{2 \pi i}(2\pi)^{{q \o 2}-{1\o 2}\sum_{i=1}^{q}k_i}\left( \prod_{i=1}^qk_i^{{\bar{\beta_i}} -{1\o 2}}\right)  \int\limits_{\left( d/ L\right) } \left({x^{  L}\o \prod_{i=1}^qk_i^{k_i} }\right)^{-z} 			{\prod_{i=1}^q\prod_{j=1}^{k_i}\Gamma\left(z+{\bar{\beta_i}\o k_i} +{j-1\o k_i}\right)dz}.
		\end{align}
		The above expression agrees with  that for a Meijer $G$-function. In particular, since $d/L>0$, the line of integration does separate the poles as prescribed in \eqref{meijer G}. Putting 
		\begin{align}\label{defbij}
			b_{ij}={\bar{\beta_i}\o k_i} +{j-1\o k_i}, \quad  \ 1\leq j\leq k_i, \  1\leq i\leq q,
		\end{align}
		and $w={x^{  L}\o \prod_{i=1}^qk_i^{k_i} }$, we find that 
	\begin{align}
			\tilde{Z}_{\boldsymbol{\alpha,\beta} }(x)&=
			{L}(2\pi)^{{q \o 2}-{1\o 2}\sum_{i=1}^{q}k_i}\left( \prod_{i=1}^qk_i^{{\bar{\beta_i}} -{1\o 2}}\right)  \frac{1}{2 \pi i}\int\limits_{\left( d/  L\right) } w^{-z} 			{\prod_{i=1}^q\prod_{j=1}^{k_i}\Gamma\left(z+b_{ij}\right)dz}\\
			&= {L}(2\pi)^{{q \o 2}-{1\o 2}\sum_{i=1}^{q}k_i}\left( \prod_{i=1}^qk_i^{{\bar{\beta_i}} -{1\o 2}}\right) 
			G_{0 \ , \ \ \sum k_i}^{ \sum k_i, \ 0}\left(\begin{matrix}
				&\rule{0.5cm}{0.05cm} \\ & b_{i,j}
			\end{matrix} \ \ \bigg| w  \right)  .\label{ztilde in G}
		\end{align}
		We now use an asymptotic relation for the Meijer $G$-function from \cite[p. 180]{luke}. One can compute that in our case, $\s =\sum_{k=1}^{q}k_i$ and $\theta={1\o \s}\left(\sum_{i=1}^{q}\bar{\beta_i}+{1-q \o 2}\right)$, where $\s$ and $\theta$ are defined as per the notation on p. 180 of \cite{luke}.  We get 
		\begin{align}
			\tilde{Z}_{\boldsymbol{\alpha,\beta} }(x)&\ll_F	G_{0 \ , \ \ \sum k_i}^{ \sum k_i, \ 0}\left(\begin{matrix}
				&\rule{0.5cm}{0.05cm} \\ & b_{i,j}
			\end{matrix} \ \ \bigg| w  \right)\\
			&\ll_F \frac{(2\pi)^{\sum k_i-1 \o 2}}{\sqrt{\sum k_i}} \exp\left( -\sum k_i w^{1/ \sum k_i}\right) w^\theta 
			\label{estimate decay Z}
			\\
			&\ll_F  \exp\left( -\sum k_i \left({x^{  L}\o \prod_{i=1}^qk_i^{k_i} }\right)^{1/\sum k_i}\right) x^{  L {\left(\sum_{i=1}^{q}\bar{\beta_i}+{1-q \o 2}\right)}/\sum k_i}. 
		\end{align}
		Observing that 
		\begin{align}
\sum_{i=1}^q k_i = \sum_{i=1}^q \alpha_i L = {Ld_F \o 2  }, 
		\end{align}		
		we complete the proof.
	\end{proof}
	
	\begin{lemma} \label{lem:Z tilda dash bd}
		Let $\boldsymbol{\a}=(\a_1,\cdots, \a_q), \ \boldsymbol{\b}=(\b_1,\cdots, \b_q),$ with  $\a_i>0$ and $\Re(\b_i)>0$ for all $i$. Assume that each component of  $\boldsymbol{\a}$ is rational, say $\a_i={m_i\o n_i}$ with $(m_i, n_i)=1$. Let $L$ be the lcm $[n_1, \cdots , n_q]$. Denoting the positive integer $\a_i L$ by $k_i$, we define  $C_i$'s as in \eqref{Ci}. Then as $x\to \infty $,  we have 
		\begin{align}
			\tilde{Z}'_{\boldsymbol{\alpha,\beta} }\left( { x}\right) \ll  \exp(-C_1x^{C_2}) x^{C_2C_3+C_4},
		\end{align}
		where $C_4=\max\{0, C_2-1\}$.
	\end{lemma}
	\begin{proof}
		We will use the identity (12) from Luke \cite[p. 151]{luke}, which gives after some simplification    
		\begin{align}\label{derivative z tilde}
			{d\o dz}\left(G_{0 \ , \ \ \sum k_i}^{ \sum k_i, \ 0}\left(\begin{matrix}
				&\rule{0.5cm}{0.05cm} \\ & b_{i,j}
			\end{matrix} \ \ \bigg| z  \right)\right)
		=
		{b_{11}\o z}G_{0 \ , \ \ \sum k_i}^{ \sum k_i, \ 0}\left(\begin{matrix}
				&\rule{0.5cm}{0.05cm} \\ & b_{i,j}
			\end{matrix} \ \ \bigg| z  \right)-{1\o z}G_{0 \ , \ \ \sum k_i}^{ \sum k_i, \ 0}\left(\begin{matrix}
				&\rule{0.5cm}{0.05cm} \\ & b_{11}+1, b^{*}_{i,j}
			\end{matrix} \ \ \bigg| z  \right),
		\end{align}
		where the $b_{ij}$'s  are as defined in \eqref{defbij} and $b^{*}_{ij}$ denotes the same sequence as $b_{ij}$ without the first term $b_{11}$. From \eqref{ztilde in G}, letting $z={x^{  L}\o \prod_{i=1}^qk_i^{k_i} }$ and $L_1=  {L}(2\pi)^{{q \o 2}-{1\o 2}\sum_{i=1}^{q}k_i}\left( \prod_{i=1}^qk_i^{{\bar{\beta_i}} -{1\o 2}}\right)$, we have 
		\begin{align}
			\tilde{Z}'_{\boldsymbol{\alpha,\beta} }(x)&=L_1
			{d\o dz}\left(G_{0 \ , \ \ \sum k_i}^{ \sum k_i, \ 0}\left(\begin{matrix}
				&\rule{0.5cm}{0.05cm} \\ & b_{i,j}
			\end{matrix} \ \ \bigg| z  \right)\right){dz\o dx} 
			\\
			&=
		{{Lx^{  L-1}\o \prod_{i=1}^qk_i^{k_i} }}  {b_{11}\o z}L_1  G_{0 \ , \ \ \sum k_i}^{ \sum k_i, \ 0}\left(\begin{matrix}
			&\rule{0.5cm}{0.05cm} \\ & b_{i,j}
		\end{matrix} \ \ \bigg| z  \right)-  
	{{Lx^{  L-1}\o \prod_{i=1}^qk_i^{k_i} }}  
	{1\o z} 	L_1  G_{0 \ , \ \ \sum k_i}^{ \sum k_i, \ 0}\left(\begin{matrix}
			&\rule{0.5cm}{0.05cm} \\ & b_{11}+1, b^{*}_{i,j}
		\end{matrix} \ \ \bigg| z  \right), 
		\end{align} 
	using	 \eqref{derivative z tilde}.  For the first term above, we use \eqref{ztilde in G} followed by  Lemma \ref{lem:Z tilda bd}. For the second term,  following the notation of \cite[p. 180]{luke},   the parameter  $\sigma = \sum k_i $ while  $\theta$ takes the value
	${1\o \sum k_i}\left(\sum_{i=1}^{q}\bar{\beta_i}+1+{1-q \o 2}\right)$. We then apply the estimate \eqref{estimate decay Z}. 
We have thus obtained 
		\begin{align}
			\tilde{Z}'_{\boldsymbol{\alpha,\beta} }(x)&\ll_F\left({\bar{\b_1}\o k_1}\exp \big( {-C_1x^{C_2}} \big) x^{C_2C_3}
			+{L\o x} \exp \big( {-C_1x^{C_2}} \big) x^{C_2(C_3+1)}\right)\\
			&\ll_F \exp \big( {-C_1x^{C_2}} \big) x^{C_2C_3+C_4}.
		\end{align} 
	\end{proof}

	\section{Ramanujan-Hardy-Littlewood-type identity for Selberg class}\label{mr}
	Before giving the proof of Theorem \ref{rhlselbergth}, we prove the following results.
	We will first derive an approximate formula for ${F'\over F }(s)$ in terms of zeros $\rho$ of $F$ which are near $s$. This is a generalization of Theorem 9.6(A) of \cite{titch}. 
	\begin{lemma}
		Let $F$ satisfy axioms (i)-(iv) of the class $\mathcal{S}$. If $\rho =\beta +i\gamma$ runs through zeros of $F(s)$, 
		\begin{align}\label{lemma1}
			\frac{F' (s)}{F (s)} ={ \sum_{|t-\gamma| \leq 1} \frac{1}{s-\rho}}+{O_F}(\log t).
		\end{align}
		uniformly for $\sigma\in [-c_F,1+c_F]$, where $s=\s+it$ and  $c_F$ is as defined in \eqref{defcF}.
	\end{lemma}
	\begin{proof}
		Let $s=\sigma +it$ with  $t>2$ and $\sigma\in  [-c_F,1+c_F]$. We will use Lemma $(\alpha)$ on p. 56 of \cite{titch} with $f(s)=F(s)$, $s_0=2+it$ and $r=4(2+c_F)$. Then from lemma 3.4 (in particular (5) and (9)) of \cite{Dixit and murty}, it is clear that the hypothesis 
		\begin{align}
			\left|{f(s)\over f(s_0)}\right|<e^M
		\end{align}
		holds with $M=A\log t$ for some constant $A$ depending on $F$. Lemma $(\alpha)$ of \cite[p. 56]{titch} then yields 
		\begin{align}\label{approx1}
			{F'(s)\over F(s)}=\sum_{ \rho :|\rho-s_0|\leq 4+2c_F}{1\over s-\rho}+O_F(\log t)
		\end{align}
		for $|\sigma -2|\leq {r\over 4}$, and so in particular for $-c_F\leq \sigma \leq 1+c_F$ due to the choice of $r$. We would like to replace the main term in \eqref{approx1} by 
		\begin{align}\label{approx2}
			\sum_{ \rho :|\gamma-t|\leq 1}{1\over s-\rho},
		\end{align}
		where $\gamma$ refers to the ordinate of the zero $\rho$. Indeed, the sum in \eqref{approx1} runs over more number of terms than in \eqref{approx2}, but since $|\rho-s_0|\leq 4+2c_F$ implies $|\gamma -t|\leq 4+2c_F$, this difference is atmost $$N_F(t+4+2c_F)-N_F(t-4-2c_F)\ll_F\log T,$$ by \eqref{dens}. This completes the proof. 
	\end{proof}

	\begin{lemma}\label{lemma}
			Let $F$ satisfy axioms (i)-(iv) of the class $\mathcal{S}$.  Suppose that the degree of $F$ is $d_F$. Let  $A_1>0$ be a sufficiently small constant.	Let $T\rightarrow \infty$ through values such that 
		$|T-\gamma|>\exp{ \left(-A_1 \gamma / \log\gamma \right)}$ 
		for every ordinate $\gamma $ of a zero of $F(s)$.
		Then for $\sigma \in [-c_F,1+c_F ]$, 
		where $c_F$ is as defined in \eqref{defcF}, we have  
		\begin{equation}
			|F(\sigma + iT)| \geq e^{-A_2 T},	
		\end{equation}
		where $0<A_2 < {\pi d_F\over4}$.
	\end{lemma}	
	\begin{proof}
		Our starting point for the proof is \eqref{lemma1}, which holds uniformly for $s=\sigma +it$ with $\sigma \in [-c_F, 1+c_F]$.
		Integrating \eqref{lemma1} with respect to $s$ from ${1+c_F} +it$ to $z=\sigma' + it$, where $\sigma' \in [-c_F,{1+c_F}]$, and $t$ does not equal the ordinate of any zero  of $F(s)$, we get
		\begin{align}\label{eq7}
			\log F (z) - \log F \left({{1+c_F}} +it \right)&=\bigg[{\sum_{|t-\gamma| \leq 1} \log (s-\rho) } \bigg]_{{{1+c_F}} +it} ^{{z}} +{O}_F(\log t)  \\
			&={\sum_{|t-\gamma| \leq 1} \log (z-\rho) }-{\sum_{|t-\gamma| \leq 1} \log \left( {{1+c_F}}+it-\rho\right)  } +{O}_F(\log t).
		\end{align}
		Since  $\log \left({1+c_F}+it-\rho\right) $ is bounded when $|t-\gamma|\leq1$, and the number of terms in the second sum above is $\ll \log  t$  from \eqref{dens}, we obtain
		\begin{equation}
			{\displaystyle\sum_{|t-\gamma| \leq 1} \log \left({1+c_F}+it-\rho \right) }={O}_F(\log t).	
		\end{equation} 
		Moreover, 
			\begin{equation}
				| F\left({1+c_F}+it\right)|\leq \sum_{n=1}^{\infty}\frac{|a(n)|}{n^{1+c_F}}, 
			\end{equation} 
		which is bounded due to absolute convergence of $F(s)$ to the right of $\Re (s)=1$. 
		Consequently, \eqref{eq7} yields
		\begin{align}
			\log F (s) ={\sum_{|t-\gamma| \leq 1} \log (s-\rho) }+{O}_F(\log t).
		\end{align}
		Since Re$(s-\rho)$ is bounded, taking real parts on both sides, we obtain
		\begin{align}
			\log |F (s)| &={\sum_{|t-\gamma| \leq 1} \log |s-\rho| }+{O}_F(\log t)\\
			&\geq {\sum_{|t-\gamma| \leq 1} \log |t-\gamma| }+{O}_F(\log t).
		\end{align}
		Let $A_1$ be a sufficiently small constant to be chosen later. We now let $s=\sigma + iT$, where $T\to \infty$ through values such that $|T-\gamma | > \exp\left(-A_1\gamma/ \log \gamma \right)$ for every ordinate $\gamma$ of a zero of $F(s)$. Then the above inequality yields 
		\begin{align}\label{lower1}
			\log |F (\sigma + iT)| \geq -{\sum_{|T-\gamma| \leq 1}{ A_1 \gamma \over \log\gamma }}
			+{O}_F(\log T).
		\end{align}
		From \eqref{dens}, we see that
		\begin{align}\label{lower2}
			{\sum_{|T-\gamma| \leq 1} {A_1 \gamma \over \log\gamma }} \leq \sum_{|T-\gamma| \leq 1} A_1 \frac{T+1}{\log(T-1)} \leq C_1A_1T,
		\end{align}
		for some absolute constant $C_1>0$. 
		Equations \eqref{lower1} and \eqref{lower2} imply that
		\begin{align} \label{log }
			\log |F (\sigma + iT)| \geq -C_2A_1T,
		\end{align}
	for some absolute constant $C_2>0$.
		Choosing $A_1$ sufficiently small so that $C_2A_1< {\pi d_F\over4}$ completes the proof.
	\end{proof}

	\subsection{Proof of Theorem \ref{rhlselbergth}}
	From \eqref{defZ}, we have
	\begin{align}
		\sum\limits_{n = 1}^\infty  {\frac{{b_F(n)}}{n}{Z_{\boldsymbol{\a,\b}}}\left( {\frac{\eta }{n}} \right)}  &= \sum\limits_{n = 1}^\infty  {\frac{{b_F(n)}}{n}\frac{1}{{2\pi i}}\int\limits_{(c)} {\prod_{i=1}^q\Gamma\left(\alpha_i s+\bar{\beta_i} \right){{\left( {\frac{\eta }{n}} \right)}^{ - s}}ds} }    \\	
		& = \frac{1}{2\pi i}  \int\limits_{(c)} \prod_{i=1}^q\Gamma\left(\alpha_i s+\bar{\beta_i} \right)\sum\limits_{n = 1}^\infty\frac{b_F(n)}{n^{1-s}}{ {\eta }}^{ - s}ds   \\   	
		&   = \frac{1}{{2\pi i}}\int\limits_{\left( {c} \right)} {\frac{\prod_{i=1}^q\Gamma\left(\alpha_i s+\bar{\beta_i} \right)}{{F(1 - s)}}{\eta ^{ - s}}ds}.   
	\end{align}
	In the second step above, we have interchanged the order of integration and summation by using Stirling's formula \eqref{strivert}.
	Using the functional equation \eqref{Functional equation}, we obtain
	\begin{align}\label{xr2}
		\sum\limits_{n = 1}^\infty  {\frac{{b_F(n)}}{n}{Z_{\boldsymbol{\a,\b}}}\left( {\frac{\eta }{n}} \right)}  &= \frac{\bar{\omega}}{{2\pi i}}\int\limits_{\left( {c} \right)} Q^{1-2s}\frac{\prod_{i=1}^q\Gamma\left(\alpha_i(1-s)+{\beta_i} \right)}{\overline{{F(\bar{s})}}}{\eta ^{ - s}}ds.	
	\end{align} 
	\begin{figure}
		\centering
       	\includegraphics[scale=0.22,bb=100 100 1000 1000]{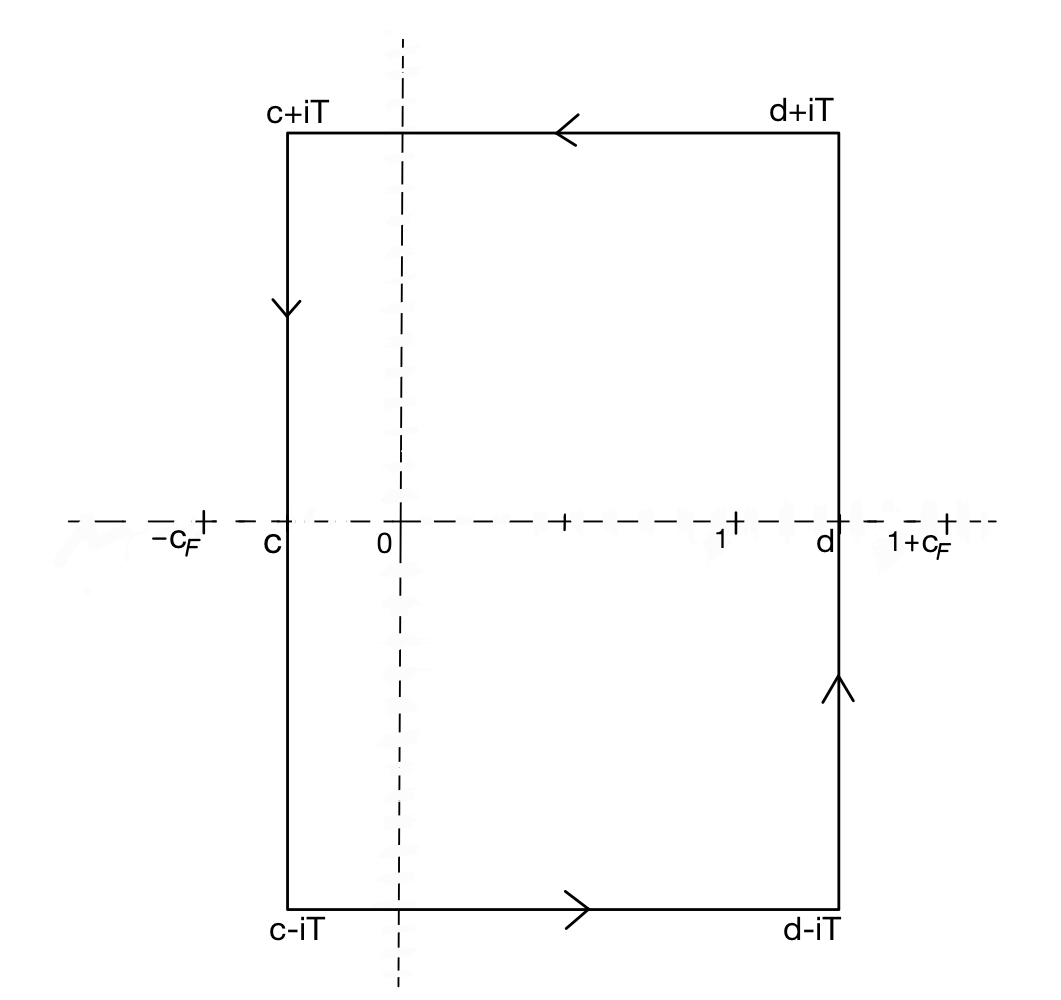}
		\vspace{0.7cm}
			\caption{ The contour $\mathscr{C}$}\label{fig:Contour}
	\end{figure}
We first assume that $r=j_F-k_F>0$. As seen in Figure 1 below, we consider the rectangular contour $\mathscr{C}$ with sides $\left[c -iT, d-iT \right],  \left[d-iT, d + iT \right], \left[d + iT, c+iT \right]$  and $ \left[c+iT, c-iT \right] $, where $T>0$ and $1<d<1+c_F $. From the choice of $c_F$ as given in \eqref{defcF} and the discussion preceding it, it is clear that the only possible poles of the $\Gamma$-factors inside $\mathscr{C}$ are at $s=1$, from those factors for which $\b_i=0$. As there are $j_F$ such factors and $F$ has a pole of order $k_F$ at $s=1$, this yields a pole of order $j_F-k_F$ at $s=1$ for the integrand in \eqref{xr2}. 
	Again from the discussion preceding \eqref{defcF}, $F(s)$	has a zero of order $j_F-k_F$ at $s=0$, thereby giving a pole of same order for our integrand at $s=0$.
	
	Finally, we also have simple poles at conjugates of non-trivial zeros of $F(s)$. Applying Cauchy's residue theorem gives 
	\begin{align}
		&\frac{1}{{2\pi i}}\left[\hspace{5pt} {\int\limits_{ c - iT}^{d - iT} {+\int\limits_{d - iT}^{d + iT} {+\int\limits_{d + iT}^{ c + iT} {+\int\limits_{ c + iT}^{ c - iT} {} } } } } \hspace{5pt}\right]  \bar{\omega} Q^{1-2s}\frac{\prod_{i=1}^q\Gamma\left(\alpha_i(1-s)+{\beta_i} \right)}{\overline{{F(\bar{s})}}}{\eta ^{ - s}}ds  \\
		& \quad =\sum\limits_\rho   {\bar{\omega}} Q^{1-2\bar{\rho}}\frac{\prod_{i=1}^q\Gamma\left(\alpha_i(1-\bar{\rho})+{\beta_i} \right)}{\overline{{F'({\rho})}}}{\eta ^{ - \bar{\rho}}} + {\left. {\frac{1}{{(r - 1)!}}\frac{{{d^{r - 1}}}}{{d{s^{r - 1}}}}{{s}^r} \bar{\omega} Q^{1-2s}\frac{\prod_{i=1}^q\Gamma\left(\alpha_i(1-s)+{\beta_i} \right)}{\overline{{F(\bar{s})}}}{\eta ^{ - s}}} \right|_{s = 0}}   \\
		&  \qquad+ {\left. {\frac{1}{{(r - 1)!}}\frac{{{d^{r - 1}}}}{{d{s^{r - 1}}}}{{\left( {s - 1} \right)}^r} \bar{\omega} Q^{1-2s}\frac{\prod_{i=1}^q\Gamma\left(\alpha_i(1-s)+{\beta_i} \right)}{\overline{{F(\bar{s})}}}{\eta ^{ - s}}} \right|_{s = 1}}.\label{xr3}
	\end{align}  	
	By Lemma \ref{lemma}, Stirling's formula \eqref{strivert} and the definition of $d_F$, as $|T|\to \infty$, 
	\begin{align}
		Q^{1-2s}\frac{\prod_{i=1}^q\Gamma\left(\alpha_i(1-s)+{\beta_i} \right)}{\overline{{F(\bar{s})}}}{\eta ^{ - s}}=O\left(e^{\left( A_2- {\pi d_F\over4}\right)|T| } \right). 
	\end{align}  
	Since $A_2 < {\pi d_F \over 4}$, we conclude that the integrals along the horizontal line segments in \eqref{xr3} go to zero as $|T|\to \infty$. 
	
	Derivation of the second integral on the left-hand side of \eqref{xr3} is as follows. Using $\eta \nu ={1\over Q^2}$, 
	\begin{align}
		\frac{1}{{2\pi i}}	\int\limits_{d - iT}^{d + iT} \bar{\omega} Q^{1-2s}\frac{\prod_{i=1}^q\Gamma\left(\alpha_i(1-s)+{\beta_i} \right)}{\overline{{F(\bar{s})}}}{\eta ^{ - s}}ds&=\frac{1}{{2\pi i}} \int\limits_{d - iT}^{d + iT} \bar{\omega} Q\frac{\prod_{i=1}^q\Gamma\left(\alpha_i(1-s)+{\beta_i} \right)}{\overline{{F(\bar{s})}}}{\nu ^s}ds.
	\end{align}
	Putting $s=1-u$ and $c'=1-d$, we obtain 
	\begin{align}
		\frac{1}{{2\pi i}}	\int\limits_{d - iT}^{d + iT} \bar{\omega} Q\frac{\prod_{i=1}^q\Gamma\left(\alpha_i(1-s)+{\beta_i} \right)}{\overline{{F(\bar{s})}}}{\nu ^s}ds&=-\frac{\bar{\omega} Q}{{2\pi i}}	\int\limits_{c' + iT}^{c' - iT} \frac{\prod_{i=1}^q\Gamma\left(\alpha_iu+{\beta_i} \right)}{\overline{{F(1-\bar{u})}}}{\nu ^{1-u}} du  \\ 
		&=\frac{\bar{\omega} Q}{{2\pi i}}	\int\limits_{c' - iT}^{c' + iT} {\prod_{i=1}^q\Gamma\left(\alpha_iu+{\beta_i} \right)}\sum_{n=1}^{\infty}{\overline{b_F(n)}\over n^{1-u}}{\nu ^{1-u}} du  \\ 
		&=\sum_{n=1}^{\infty}{\overline{b_F(n)}\over n}\frac{\bar{\omega} Q\nu}{{2\pi i}}	\int\limits_{c' - iT}^{c' + iT} {\prod_{i=1}^q\Gamma\left(\alpha_iu+{\beta_i} \right)}{\left(\nu \over n \right) ^{-u}} du,
	\end{align}
	where the second equality is valid as $\Re(1-\bar{u})=d>1$. As $T\to \infty$, since $-c_F<c'<0$,  using \eqref{defZ}, we obtain
	\begin{align}\label{xr4}
		\frac{1}{{2\pi i}}	\int\limits_{(d)} \bar{\omega} Q^{1-2s}\frac{\prod_{i=1}^q\Gamma\left(\alpha_i(1-s)+{\beta_i} \right)}{\overline{{F(\bar{s})}}}{\eta ^{ - s}}ds=\bar{\omega} Q\nu\sum_{n=1}^{\infty}{\overline{b_F(n)}\over n}Z_{\boldsymbol{\alpha,\bar{\beta}}}	\left(\nu \over n \right).
	\end{align}
	Letting $T\to \infty$ in \eqref{xr3} and using  \eqref{xr2}  and \eqref{xr4}, we get 
	\begin{align}
		&-\sum\limits_{n = 1}^\infty  {\frac{{b_F(n)}}{n}{Z_{\boldsymbol{\alpha,\beta}}}\left( {\frac{\eta }{n}} \right)} +\bar{\omega} Q\nu\sum_{n=1}^{\infty}{\overline{b_F(n)}\over n}Z_{\boldsymbol{\alpha,\bar{\beta}}}	\left(\nu \over n \right) \\
		& \quad =\sum\limits_\rho   {\bar{\omega}} Q^{1-2\bar{\rho}}\frac{\prod_{i=1}^q\Gamma\left(\alpha_i(1-\bar{\rho})+{\beta_i} \right)}{\overline{{F'({\rho})}}}{\eta ^{ - \bar{\rho}}} + {\left. {\frac{1}{{(r - 1)!}}\frac{{{d^{r - 1}}}}{{d{s^{r - 1}}}}{{s}^r} \bar{\omega} Q^{1-2s}\frac{\prod_{i=1}^q\Gamma\left(\alpha_i(1-s)+{\beta_i} \right)}{\overline{{F(\bar{s})}}}{\eta ^{ - s}}} \right|_{s = 0}}   \\
		&  \qquad+ {\left. {\frac{1}{{(r - 1)!}}\frac{{{d^{r - 1}}}}{{d{s^{r - 1}}}}{{\left( {s - 1} \right)}^r} \bar{\omega} Q^{1-2s}\frac{\prod_{i=1}^q\Gamma\left(\alpha_i(1-s)+{\beta_i} \right)}{\overline{{F(\bar{s})}}}{\eta ^{ - s}}} \right|_{s = 1}}.
	\end{align}
	Multiplying both sides by $-\omega \sqrt{\eta} $ and using the fact $\eta \nu={1\over Q^2}$, we get the required result \eqref{rhlselbergeq}. 
	
	For $r\leq 0$, it is easy to see that last two terms on the right-hand side of \eqref{xr3} do not appear, but the remainder of the proof remains unchanged.
	This completes the proof.

		\section{Proofs of Corollaries  \ref{coro1} and \ref{coro2}}
\label{sec:cor}
In this section, we prove some consequences of Theorem \ref{rhlselbergth}.   
\subsection{Proof of Corollary \ref{coro1}.}
		We will use the notation set up in the statement of  Lemma \ref{lem:Z tilda bd}.  From the data of $F$, we have $q=1$, $\bs \a=(4)$, $\bs \b = (\b_1)$,  hence  $L=1$ and $k_1=4$.  The sequence $b_{ij}$ of \eqref{defbij} is given by  
	$
		 b_{1,j} =( \bar{\b_1} + j-1 )/ 4$  with $ j =1, \hdots, 4.  
	$
We  	express  ${Z}_{{(4),(\beta_1)} }(x)$ in terms of  $\tilde{Z}_{{(4),(\beta_1)} }(x)$ using \eqref{Z vs Z tilda} and then write $\tilde{Z}_{{(4),(\beta_1)} }(x)$ in terms of the Meijer $G$-function using \eqref{ztilde in G} to obtain 
			\begin{align}
		{Z}_{{(4),(\beta_1)} }(x) 
			&= \begin{cases}
				(2\pi)^{-{3\over 2}}
				\frac{4^{\bar{\b}_1}} {2} 
				G^{4 \  0}_{0 \ 4}\Bigg( \left. \begin{matrix}
					\rule{.4cm}{.2mm} \\ {\bar{\b_1}\over 4} \ {\bar{\b_1}\over 4}+{1\over 4} \  {\bar{\b_1}\over 4}+{2\over 4}\ {\bar{\b_1}\over 4}+{3\over 4}
				\end{matrix} \right| {x \over 4^4} \Bigg)& \text{if } \bar{\b}_1 \neq 0\\
				{(2\pi)^{-{3\o 2}} } {1  \o 2}G^{4 \  0}_{0 \ 4}\Bigg( \left. \begin{matrix}
					\rule{.4cm}{.2mm} \\ 0 \ {1\over 4} \  {2\over 4}\ {3\over 4}
				\end{matrix} \right| {x \over 4^4} \Bigg)-{1\o 4}& \text{if } \bar{\b}_1 = 0.
			\end{cases}
		\end{align}
For this special  value of the Meijer G-function,  \eqref{sepG1} gives 
		\begin{align}
		{Z}_{{(4),(\beta_1)} }(x)=
			\begin{cases}
				{1\o 4}x^{\bar{\b}_1/ 4}e^{- {x}^{1/4}  }& \text{if } \b_1 \neq 0\\
				{1\o 4}e^{-x^{1/4}  }-{1\o 4} & \text{if } \b_1 = 0.
			\end{cases}
		\end{align}
	Substituting these values into Theorem \ref{rhlselbergth}, we obtain \eqref{coro1eqn} and \eqref{coro2eqn} respectively as desired.
\qed
	
\subsection{Proof of Corollary \ref{coro2}.}	Again, following the notation of Lemma \ref{lem:Z tilda bd}, in this case we have $q=4$, $L=1$ and $k_i=1 $ for $1\le i \le 4$.  Calculating the residue term of  \eqref{Z vs Z tilda}, it can be seen that 
\begin{align} \label{Z to Z tilda}
	{Z}_{{(1,1,1,1),(0,{1\o 2},-C,C)} }\left(x\right) &= \tilde{Z}_{(1,1,1,1),(0,{1\o 2},-C,C) }\left(x \right)-  \Gamma(1/ 2) \Gamma(-C)  \Gamma(C). 
\end{align} 
We note that the   sequence $b_{ij}$ of \eqref{defbij} is now given by  
$
\left\lbrace b_{i,1} \right\rbrace_{i=1}^4  = \{0, 1/2, -C, C \}.  
$
Using \ref{ztilde in G}, we have 
\begin{align} \label{Z tilda to G}
\tilde{Z }_{(1,1,1,1),(0,{1 \o 2},-C,C) }\left(x \right) = 	G^{4 \  0}_{0 \ 4}\left( \left. \begin{matrix}
	\rule{.4cm}{.2mm} \\ 0 \ {1\over 2} \  -C\ C
\end{matrix} \right| x \right). 
\end{align}
Putting $b=0$ and $c=C$  in \eqref{sepG2}, we get
		\begin{align} \label{G to K}
			G^{4 \  0}_{0 \ 4}\left( \left. \begin{matrix}
				\rule{.4cm}{.2mm} \\ 0 \ {1\over 2} \  -C\ C
			\end{matrix} \right| z \right)&=8{\sqrt{\pi}} K_{2C}(2\sqrt{2}(-z)^{1/4} )K_{2C}\left({2\sqrt{2}\sqrt{z}\o (-z)^{1/4}}\right),
		\end{align}
Using the reflection formula to simplify the residue term in  \eqref{Z to Z tilda} and combining it with \eqref{Z tilda to G} and \eqref{G to K}, we obtain 
	\begin{align}
			{Z}_{{(1,1,1,1),(0,{1\o 2},-C,C)} }\left(x \right) 
			&=8{\sqrt{\pi}} K_{2C}(2\sqrt{2} ({-x})^{1/4} )K_{2C}\left({2\sqrt{2}\sqrt{x}\o ({-x})^{1/4}}\right)+{\pi\sqrt{\pi}\o C\sin \pi C}. 
		\end{align}
	Substituting this into Theorem \ref{rhlselbergth} yields \eqref{coro4eqn}.

	\section{Riesz-type criteria for Selberg class}

	\subsection{Proof of Theorem \ref{Rhlselberg}}	
	For $F\in \mathcal{S^*}$, consider the summatory function of the coefficients of $F(s)^{-1}$, given by 
	\begin{align}
		M_F(x):=\sum_{n\leq x } b_F(n) .
	\end{align}
It can be checked that $F\in \mathcal{S^*}$ is contained in the general class of $L$-functions constructed on p.94 of Iwaniec-Kowalski \cite{Iwanic}. In particular, axioms (iv)$'$
 and the inequality \eqref{inequality} play a crucial role here. We may thus apply Proposition 5.14 of \cite{Iwanic} to obtain that GRH implies the bound
	\begin{align}\label{estmf}
		M_F(x)\ll_{\d} x^{{1\over 2}+\d}, 
	\end{align}
	for any $\d >0$, as $x\to \infty$. Let us define 
	\begin{align} \label{MFhn}
		M_F(h, N)=\sum_{n=h}^{N}{b_F(n)\over n}.
	\end{align}
	Applying partial summation to \eqref{estmf}, one can deduce that as $h\to \infty$, we have under GRH 
	\begin{align} \label{bd on MFhn}
		M_F(h, n)\ll_\d h^{-{1\over 2}+\d}, 
	\end{align}
for any $\d>0$, uniformly for any $N\geq h$. If $N<h$, we interpret $M_F (h,N)$ to be zero.
 We now commence the proof of Theorem \ref{Rhlselberg}.
 \subsubsection{Proof of Theorem \ref{Rhlselberg} (a)}
	Let $\e>0$ be sufficiently small. From \eqref{defP}, we may write  
	\begin{align}
		\mathcal{P} _{\boldsymbol{\alpha,\beta} ,z}(\nu^2)&=\sum\limits_{n = 1}^\infty  \frac{{b_F(n)}}{n}{Z_{\bs{ \a,\b}  }}\left(\left( {\frac{\nu }{n}} \right)^{d_F}\right)\cosh\left({\nu z}\over n\right)\\
		&=\left[ \sum\limits_{n = 1}^{h-1}+\sum\limits_{n = h}^{\infty} \right]  \frac{{b_F(n)}}{n}{Z_{\bs{ \a,\b}  }}\left(\left( {\frac{\nu }{n}} \right)^{d_F}\right)\cosh\left({\nu z}\over n\right)\\
		&=: P_1+P_2, \text{(say)}\label{P1+P2}
	\end{align}
where $h$ equals the greatest integer less than or equal to $\nu^{1-\e}$, that is,  $h =[\nu ^{1-\e}]$. We first consider the partial sum of $P_2$, given by 
	\begin{align}
P_2(N) &:=		\sum\limits_{n = h}^{N}\frac{{b_F(n)}}{n}{Z_{\boldsymbol{\alpha,\beta} }}\left( \left( {\frac{\nu }{n}}  \right)^{d_F} \right) \cosh \left(\frac{\nu z}{n} \right)
\\
&= \sum_{ n=h }^N\left( M_{F}(h,n)-M_F(h,n-1)\right) {Z_{\boldsymbol{\alpha,\beta} }} \left(  \left( {\frac{\nu }{n}}  \right)^{d_F}\right) \cosh \left(\frac{\nu z}{n} \right)
\\
		&=\sum_{ n=h }^{N-1} M_F(h,n) \left( f(n)-f(n+1)   \right) + M_{F}{(h,N)}f(N),  
		\label{A}
	\end{align}
	where 
	\begin{equation}
f(t) := Z_{\boldsymbol{\alpha,\beta} }\left(  \left( {\frac{\nu }{t}}  \right)^{d_F}\right) \cosh \left(\frac{\nu z}{t} \right). 
	\end{equation}
	Using Lemma \ref{lem:Z poly bound} and the bounds \eqref{bd on MFhn} and $\cosh(xz) \ll_z 1$ as $x \to 0$, we find that 
	\begin{equation}
		\label{bd for M times f}
M_F(h,N) f(N) \ll_{\d, F,z}  h^{-\frac{1}{2}+\d} \left( \frac{N}{\nu} \right)^{cd_F}, 
\label{B}
	\end{equation}
	where $-c_F<c <0$. Letting $N \to \infty$, the above expression goes to zero, so that  \eqref{A} yields
	\begin{align}
P_2	&=\sum_{ n=h }^{\infty } M_F(h,n) \left( f(n)-f(n+1)   \right)
\label{P2}
\\
&= \left[\sum_{ n=h }^{[C\nu] }+ \sum_{ n=[C\nu] }^{\infty }  \right] M_F(h,n) \left(f(n)- f(n+1)\right)  
\\
&=: P_3+P_4  \quad \text{ (say)}. 
\label{P3,P4}
	\end{align}
	Here, $C$ is a constant to be chosen later and $[C\nu]$ denotes the greatest integer less than or equal to $C\nu$. 
	Using the Mean Value theorem, there exists  $\lambda_n \in (n, n+1)$ such that 
	\begin{align} \label{MVT}
f(n) - f(n+1) &= f'(\lambda_n)  
\\
&= 
-\frac{d_F \nu^{d_F}}{\lambda_n^{d_F+1}} Z'_{\boldsymbol{\alpha,\beta} }\left(  \left( {\frac{\nu }{\lambda_n}}  \right)^{d_F}\right) \cosh \left(\frac{\nu z}{\lambda_n} \right) 
-\frac{\nu z}{\lambda_n^{2}} Z_{\boldsymbol{\alpha,\beta} }\left(  \left( {\frac{\nu }{\lambda_n}}  \right)^{d_F}\right) \sinh \left(\frac{\nu z}{\lambda_n} \right). 
	\end{align}
	We  will estimate $f(n)-f(n+1)$ in two different ranges of $n$, as relevant for $P_3$ and $P_4$ in what follows. 
	
	At this point, the assumption that $j_F=0$  plays a crucial role. Indeed, when $j_F=0$, one sees from \eqref{Z vs Z tilda} that $Z_{\boldsymbol{\alpha,\beta} }$ is the same as $\tilde Z_{\boldsymbol{\alpha,\beta} }$. Thus, for 
	 $\nu/\lambda_n$ sufficiently large, we can apply Lemmas \ref{lem:Z tilda bd} and \ref{lem:Z tilda dash bd}. Since $\lambda_n \asymp n$, this means that in the range $n \le [C\nu]$, for  sufficiently small  $C$,  we have 
	\begin{align}
f(n)-f(n+1) \ll_{F,z,C} \exp\left(-C_1 \left( \frac{\nu}{n} \right)^2 + \frac{\nu |\Re(z)| }{n} \right) \left(\frac{\nu}{n}\right)^{2C_3+d_FC_4} 
\left(  \frac{\nu^{d_F}}{n^{d_F+1}} + \frac{\nu}{n^2} \right), 
\label{f(n)1}
	\end{align}
where $C_i$'s are as defined in \eqref{Ci}. 
Here, we have used the fact that	$d_FC_2=2$ and that  $\cosh(xz)$ and $\sinh(xz)$ are trivially of the order $\exp(x \Re(z))$. Using the above estimate and \eqref{bd on MFhn} in \eqref{P3,P4}, we obtain 
\begin{align}
P_3\ll_{\d, F,z, C} \sum_{n=h}^{[C\nu]} \frac{h^{-{1 \o 2} +\d}}{n}\exp\left(-C_1 \left( \frac{\nu}{n} \right)^2 + \frac{\nu \left|\Re(z)\right| }{n} \right) 
\left(\frac{\nu}{n}\right)^{2C_3+d_FC_4 + d_F}. 
\label{P31}
\end{align}
	Keeping in mind that $C_1>0$, it can be  seen that the argument of $\exp$ is negative if $n < C_1\nu / |\Re z|$. In the  range of $n$ applicable to $P_3$, it is enough to ensure that 
	\begin{align}
C < \frac{C_1}{|\Re z|}. 
\label{cond on C}
	\end{align}
	We choose $C$ to be sufficiently small, say 
$
		C= {C_1}/({MC_1+|\Re z|}), 
$
	where $M>0$ is a sufficiently large fixed constant, so that \eqref{f(n)1} and \eqref{cond on C} are both satisfied. Writing $2C_3+d_FC_4+d_F$ as $C_5$, \eqref{P31} then gives
\begin{align}
P_3 &\ll_{\d, F, z, C} h^{-{1\o 2} +\d} \sum_{ n=h}^{[C\nu]} \frac{\nu ^{C_5}}{n^{C_5+1}} 
\ll_{\d, F, z, C} h^{-{1\o 2} +\d} \nu ^{C_5}  \int_{h}^{C\nu} \frac{1}{t^{C_5+1}} dt. 
\label{int test}
\end{align}
Thus,  we obtain 
\begin{align}
P_3\ll_{\d, F, z, C}  h^{-{1\o 2} +\d} \nu ^{C_5} h^{-C_5}
\ll_{\d, F, z, C} h^{-{1\o 2} +\d}, 
\label{bd for P3}
\end{align}
for any $\delta>0$, keeping in mind that $h$ was chosen to be $[\nu^{1-\e}]$. 

We now turn to $P_4$. Since the range of $n$ is now $n > [C\nu]$, we will use Lemmas \ref{lem:Z poly bound} and \ref{lem:Z' bd} to bound 
$Z_{\boldsymbol{\alpha,\beta} }$ and $Z'_{\boldsymbol{\alpha,\beta} }$ respectively. Moreover, we have $\cosh(xz) \ll_z 1$ and $\sinh(xz) \ll_z
x$ when $x>0$ is bounded above by some absolute constant. Applying all this to \eqref{MVT}  gives for $n > [C\nu]$, 
\begin{align}\label{f(n)2}
f(n) - f(n+1) &\ll_{F,z,C} \frac{\nu^{d_F}}{n^{d_F+1}} \left( \frac{\nu}{n} \right)^{(-c-1)d_F} + \frac{\nu}{n^2}\left( \frac{\nu}{n} \right)^{-cd_F+1} 
\\
&\ll_{F,z,C} \frac{1}{n} \left( \frac{\nu}{n} \right)^{-cd_F+2}. 
\end{align}
Again, from \eqref{P3,P4}, we  have  
\begin{align}
	\label{bd for P4}
P_4 &\ll_{\e, F,z,C}h^{-{1 \o 2} +\e} \sum_{ n=h }^\infty \frac{\nu^{-cd_F+2}  }{n^{-cd_F+3}}
\ll_{\d, F,z,C} h^{-{1\o 2} +\d},
\end{align}
 for any $\d >0$, in the same manner as done for $P_3$ in  \eqref{int test}. 
Combining \eqref{P3,P4}, \eqref{bd for P3} and \eqref{bd for P4},  we  obtain  that for any $\d>0$, 
\begin{equation}
P_2 \ll_{\d, F, z, C} h^{-{1\o2}+\d}.
\label{bd for P2}
\end{equation}

 We now consider the first series of \eqref{P1+P2}. Again, since $j_F=0$, we have $Z_{\boldsymbol{\alpha,\beta} } = \tilde Z_{\boldsymbol{\alpha,\beta} }$. Moreover, we have assumed $F\in \mathcal S^*$, which means that Lemma  \ref{lem:Z tilda bd} can be applied. This bound gives 
 \begin{align}
P_1 \ll   \sum\limits_{n = 1}^{h-1}{|b_F(n)| \over n}\exp \left(-C_1\left(\nu\over n\right)^{2} + \frac{\nu|\Re z| }{n} \right)\left(\nu\over n\right)^{2C_3}. 
 \end{align}
Here we have used the trivial bound  $\cosh (\nu z/n) \ll \exp( \nu |\Re z|/n )$.  Since $b_F$ is the inverse of $a_F$ under Dirichlet convolution, one can derive  the following relation for any prime $p$ and $m \in \mathbb N$, 
\begin{align}
b_F(p^m) = - \sum_{j=1}^{m} a_F(p^j) b_F(p^{m-j}). 
\end{align}
Using the Ramanujan bound (axiom (v)) for $a_F(n)$ and induction on $m$, it follows that  axiom (v) is also true for $b_F$ on  prime powers. As $b_F(n)$ is multiplicative, this gives $b_F(n) \ll_\d n^\d$ for any $\d >0$.  
Hence we may write
	\begin{align}
		P_1 
	& \ll \sum\limits_{n = 1}^{h-1}\left(\nu\over n\right)^{2C_3}\exp \left(-{\nu\over n}\left( C_1{\nu\over n} -{|\Re z|}\right)\right).
	\end{align}

We choose the parameter $\nu$ to satisfy  $\nu^{\e/2} >{|\Re z|\over C_1}$.  
Then for  $n$ in the range specified by the above sum and $  h =  [\nu^{1-\e}]$,  we have  $C_1{\nu\over n} -{|\Re z|}\geq C_1\left( {\nu\over h} -{\nu^{\e/2}}\right) \gg C_1\nu^\e$. Hence one has 
\begin{align}
	P_1&\ll\sum\limits_{n = 1}^{h-1} \exp \left(-{\nu\over h} C_1{\nu^\e}\right)\left(\nu\over n\right)^{2C_3}
\ll \sum\limits_{n = 1}^{h-1} \exp \left(- C_1{\nu^{2\e}}\right)\left(\nu\over n\right)^{2C_3} 
\ll \exp \left(-c{\nu^{2\e}}\right),
\label{bd for P1}
\end{align} 
for some constant $c>0$. 
Combining \eqref{bd for P2}, \eqref{bd for P1}, 
replacing $\nu^2$ by $y$, and letting  $y\to \infty$, we obtain the required bound for $\mathcal{P} _{\boldsymbol{\alpha,\beta} ,z}(y)$.
	  \\
	  
 \subsubsection{Proof of Theorem \ref{Rhlselberg} (b)}
Whenever $j_F\neq 0$,  $	{Z_{\bs{ \a,\b}  }}$ and $	{\tilde Z_{\bs{ \a,\b}  }}$ differ by  the residue term given in \eqref{Z vs Z tilda}. We first write $\mathcal{P}_{\bs {\a,\b},0}(\nu^2) = P_1+P_2$  as in  \eqref{P1+P2}. 
   Using \eqref{P2} and \eqref{MVT}, keeping in mind that  $z=0$, we can  write  
\begin{align}
P_2= \sum_{n=h}^{\infty}   -\frac{d_F\nu^{d_F}}{\lambda_n^{d_F+1}} M_F(h, n) 	{Z'_{\bs{ \a,\b}  }}\left( \left( \nu \o \lambda_n \right)^{d_F} \right). 
\end{align}
The primary departure here from the proof of part (a) given above is that due to the non-trivial residue term in \eqref{Z vs Z tilda}, we no longer have an  exponential decay bound  for $Z'_{\bs{ \a,\b}  }$. We resort to Lemma \ref{lem:Z' bd} instead. This along with \eqref{bd on MFhn} (which is true because of GRH) yields  
\begin{align}
P_2 \ll_{\d, F} h^{-{1\o 2}+\d}  \sum_{ n=h}^\infty  \frac{\nu^{d_F}}{n^{d_F+1}} \left( \frac{\nu}{n} \right)^{(-c-1)d_F}  \ll_{\d, F} h^{-{1\o 2}+\d}  \sum_{ n=h}^\infty   \frac{\nu^{-cd_F}}{n^{-cd_F+1}}. 
\end{align}
Using the integral test as done in \eqref{int test}, we find that 
$P_2 \ll_{\d, F} h^{-{1\o 2}+\d}, 
$
for any $\d >0$. 

 Turning to $P_1$, we have 
  \begin{align}
  	P_1&=\sum\limits_{n = 1}^{h-1} \frac{{b_F(n)}}{n}{Z_{\bs{ \a,\b}  }}\left(\left( {\frac{\nu }{n}} \right)^{d_F}\right)\\
  	&=\sum\limits_{n = 1}^{h-1} \frac{{b_F(n)}}{n}{\tilde{Z}_{\bs{ \a,\b}  }}\left(\left( {\frac{\nu }{n}} \right)^{d_F}\right)-\sum\limits_{n = 1}^{h-1} \frac{{b_F(n)}}{n}\Res_{s=0}\prod_{i=1}^q\Gamma\left(\alpha_i s+\bar{\beta_i} \right){\left( {\frac{\nu }{n}} \right)^{-d_Fs}}.\label{P1 in 1.4b}
  \end{align}
  The first sum of \eqref{P1 in 1.4b} satisfies the bound \eqref{bd for P1} as done in the proof of part (a), and is thus negligible compared to $P_2$.  Note that the assumption $F\in \mathcal S^*$ is  used  here. Using  $h =[\nu^{1-\e} ]$ and replacing $\nu^2$ by $y$, the  second term becomes 
  \begin{align}
  	\sum_{n=1}^{[y ^{{1\over 2}-\epsilon}] -1}{b_F(n) \over n}\Res\limits_{s=0}\prod_{i=1}^q\Gamma\left(\alpha_i s+\bar{\beta_i} \right)\left({\sqrt{y}\o n} \right)^{-d_Fs}. 
  \end{align}
 Hence GRH implies the estimate \eqref{est 1.4b}.

\subsubsection{Proof of Theorem \ref{Rhlselberg} (c)}
	
	As $z\to 0$, only the $t=0$ term survives in \eqref{lemma2}. Thus, for $0 < \Re(s)<1/2$, we have 
	\begin{align}\label{lemmaresult}
		{F(2s+1)} \int_{0}^{\infty}y^{-s-1}\mathcal{P} _{\boldsymbol{\alpha,\beta} ,0}(y)dy={2\o d_F }{\prod_{i=1}^{q}\Gamma\left( -{2\a_is\o d_F}+\bar \b_i  \right) }.
	\end{align}
We will try to extend this identity to a larger region of the complex plane. The right hand side has no poles in the region $\Re(s) \le 0$ except for a pole of order $j_F$ at $s=0$.  Since $F(2s+1)$ has a pole of order $k_F$ at $s=0$, let us multiply both  sides by $s^\ell$ where  $\ell=\max[k_F,j_F]$.  This gives 
	\begin{align} \label{identity ext}
s^\ell	{F(2s+1)} \int_{0}^{\infty}y^{-s-1}\mathcal{P} _{\boldsymbol{\alpha,\beta} ,0}(y)dy= {2 s^\ell \o d_F }{\prod_{i=1}^{q}\Gamma\left( -{2\a_is\o d_F}+\bar \b_i  \right) }.
\end{align}
We next  partition the above integral as 
\begin{align}
I:=  	\int_{0}^{\infty}y^{-s-1}\mathcal{P} _{\boldsymbol{\alpha,\beta} ,0}(y)dy=\int_{0}^{1}y^{-s-1}\mathcal{P} _{\boldsymbol{\alpha,\beta} ,0}(y)dy+\int_{1}^{\infty}y^{-s-1}\mathcal{P} _{\boldsymbol{\alpha,\beta} ,0}(y)dy =: I_1+I_2.
\end{align}

The assumed bound $\mathcal{P}_{\boldsymbol{\alpha,\beta},0 }(y) \ll_{F, \d} y^{-{1\o 4} + \d}$, implies that $I_2$ is integrable in the region $-{1\o4} < \Re(s)$.  In order to analyze $I_1$, we apply   \eqref{defP} and the  bound of Lemma \ref{lem:Z poly bound} to $Z_{\boldsymbol{\alpha,\beta}}\left( (\sqrt y/ n)^{d_F} \right)$. It can then be seen that $I_1$ is integrable for $\Re(s) < -cd_F/2$, where $-c_F<c<0$.  It thus follows that $I$ is analytic for $\Re(s) \in (-1/4, \epsilon)$ for $\e>0$ sufficiently small. 
As both sides of the identity \eqref{identity ext} are  analytic in this region, by the  principle of analytic continuation,  \eqref{identity ext} holds for an extended region of $\mathbb C$, in particular for  $-{1\over4}<\Re(s)\le0$.
	
	
 Since the right hand side of \eqref{identity ext} does not vanish in the region $-{1\o 4} <\Re s<0$ and the integral $I$ is analytic here, we see that  $F(2s+1)$ does not vanish in this region.  This implies the non-vanishing of  $F(s)$ in  ${1\over 2}<\Re(s)<1$ as needed.

\subsubsection{Proof of Theorem \ref{finite zero}}   
Let $z\ne 0 $, $\arg z \ne \pm {\pi/2}$ be  fixed. 	Let $s=\sigma {+iT}$, and $\bar{\b_i}=\delta_i+i\gamma_i$ for  $i=1, \hdots, q$. From  \eqref{lemma2}, we have for $0< \Re s <1/2$, 
\begin{align}
s^{k_F}	\label{F times phi}
	&F(2s+1) \phi(s,\bs \a, \bs \b, z )	={2 s^{k_F} \o d_F}\sum_{t=0}^{\i}{z^{2t}\over(2t)!}{\prod_{i=1}^q\Gamma\left({2\alpha_i\o d_F} (t-s)+\bar{\beta_i} \right)}, 
\end{align}	
where $  \phi(s,\bs \a, \bs \b, z )	=   \int_{0}^{\infty}y^{-s-1}\mathcal{P} _{\boldsymbol{\alpha,\beta} ,z}(y)dy$,  and $k_F$ is the order of the pole of $F(s)$ at $s=1$.  We will first extend this identity to the region $-1/4< \Re (s)\le 0$. 
Using Stirling's formula \eqref{strivert}, as  $|T|\to \infty$, the product of Gamma factors above can be written as 
\begin{align}	
	\label{after stirling}
	 (2\pi)^{q/2} \left( 1+O_{F} \left( \frac{1}{ |T|}\right) \right) \prod_{i=1}^q\bigg( \bigg|{2\alpha_iT\o d_F}-\gamma_i \bigg|^{{2\alpha_i(t-\sigma)\o d_F}+\delta_i-{1\over 2}} e^{-{\pi\over 2}\big| {2\alpha_iT\o d_F}-\gamma_i \big|}\bigg). 
\end{align}	
In particular, using \eqref{after stirling}, it is easy to check that the radius of convergence of the power series on the right hand side of \eqref{F times phi} is infinite. Using the assumed bound $\mathcal P_{\bs \a, \bs \b, z} (y)  \ll_{F, \d} y^{- {1\o 4} + \d}$, the integral $ \phi(s,\bs \a, \bs \b, z )	$  can be  shown to be analytic in the region $\Re(s) \in (-1/4,\e)$ for $\e>0$ sufficiently small, exactly as done in the proof  of Theorem \ref{Rhlselberg} (c) above. 
Moroever, the factor $s^{k_F} F(2s+1)$ is entire. From the principle of analytic continuation,  we see  that \eqref{F times phi} holds in $-1/4 < \Re(s)\le 0$.

We now consider $s=\sigma {+iT}$, where $-1/4 <  \sigma  < 0$. 
Plugging \eqref{after stirling} into \eqref{F times phi}, we obtain for $ s^{k_F} F(2s+1)\phi(s, \bs \a, \bs \b, z)$, the main term 
\begin{align}		
&{2 s^{k_F}(2\pi)^{q/2}\o d_F}\prod_{i=1}^q\bigg( \bigg|{2\alpha_iT\o d_F}-\gamma_i \bigg|^{-{2\alpha_i \sigma \o d_F}+\delta_i-{1\over 2}}   e^{-{\pi\over 2} \big|  {2\alpha_iT\o d_F}- \gamma_i \big|}\bigg)
\sum_{t=0}^{\i}{z^{2t}\over(2t)!}\prod_{i=1}^q\left|  {2\alpha_iT\o d_F} - \gamma_i \right|^{2\alpha_i t\o d_F}
\label{MT1}
\\
=	&{2 s^{k_F}(2\pi)^{q/2}\o d_F}\prod_{i=1}^q\bigg( \bigg|{2\alpha_iT\o d_F}-\gamma_i \bigg|^{-{2\alpha_i \sigma \o d_F}+\delta_i-{1\over 2}}   e^{-{\pi\over 2} \big|  {2\alpha_iT\o d_F}- \gamma_i \big|}\bigg)
	\cosh\left(z \prod_{i=1}^q\left|{2\alpha_iT\o d_F}- \gamma_i\right|^{\alpha_i\over  d_F}\right), 
	\label{MT2}
\end{align}
using  the Taylor series for $\cosh(z)$.  The error term involved is  $O_F(1/|T|)$ times the absolute value of \eqref{MT2}. 
Note that if $\arg z = \pm {\pi\o 2}$, then the argument of $\cosh$ above is  purely imaginary, which would mean that the $\cosh$ term is bounded since $\cosh z= \cos(iz)$.  Thus, we see that if $z\ne 0$ and $\arg z \ne \pm {\pi\o 2}$, then as $T\to \infty$, the $\cosh$ term is unbounded.  This means that as $T \to \infty$, the expression \eqref{MT2} tends to infinity in absolute value. Hence there must exist a sufficiently large value $T_{F,z}$ of $T$, such that for $T>T_{F,z}$,  the  expression \eqref{MT2} is non-zero. 
Since $s^{k_F}\phi(s, \bs \a, \bs \b, z)$ is analytic in this region, we must have that $F(2s+1)$ does not vanish for $s=\s+iT$ with $\s \in (-1/4,0)$ and $T> T_{F,z}$. This means that any zeros of $F(s)$ in the critical strip which are  off the  $\Re (s)=1/2$ line must lie in a vertical strip of fixed height.   As the number of such zeros is finite, this proves the result. 


\section{Concluding Remarks}
We make some remarks regarding   potential directions  of  further investigation here. 
Theorem \ref{rhlselbergth} and Corollaries \ref{coro1}, \ref{coro2}, each assume the convergence of a certain series. It is possible that this assumption can be weakened  by assuming convergence under the kind of bracketing given by \eqref{bracketing}. This may be feasible along the lines of the treatment in the  classical case. 

 Indeed, it may even be possible to prove such ``bracketed convergence" as has been done for the Riemann zeta function. This is a non-trivial question, which we have not attempted to resolve here. 

Another natural question is whether some of the hypotheses assumed to prove Riesz-type criteria for our $L$-functions  can be eliminated.  While it is known that the Euler product axiom  and the Ramanujan hypothesis are crucial  in order to consider  GRH (see for instance \cite[pp. 27–28]{perelli}), it is worth asking whether the polynomial Euler product assumption can be dispensed with. One can also enquire whether Ramanujan-Hardy-Littlewood type identities are valid for larger classes of $L$-functions, for instance the Lindel\"of class \cite{lindelof1, lindelof2}.  

In line with Lemmas \ref{lem:Z tilda bd} and \ref{lem:Z tilda dash bd}, it is possible to obtain non-trivial estimates for higher derivatives of $\tilde{Z}_{\bs \a, \bs \b}(x)$, which may find useful applications in related problems. 

\section*{Acknowledgments}
This work was supported by the Shyama Prasad Mukherjee fellowship under Council of Scientific and Industrial Research [SPM-06/1031(0281)/ 2018-EMR-I]; Scheme for Promotion of Academic and Research Collaboration under Ministry of Human Resource Development [SPARC/2018 -2019/P567/SL]; Science and Engineering Research Board-Department of Science \& Technology  [ECR/2018/001566]; and INSPIRE Faculty Award Program under Department of Science and Technology  [DST/INSPIRE/Faculty/Batch-13/2018].


\end{document}